\documentclass[11pt]{amsart}
\usepackage{ amsmath, amsthm, amsfonts, hyperref, graphicx, ifpdf}
\usepackage{mathrsfs,epsfig}
\usepackage[dvipsnames,usenames]{color}
\hypersetup{
   unicode=false,          
   pdftoolbar=true,        
   pdfmenubar=true,        
   pdffitwindow=false,     
   pdfstartview={},    
   pdftitle={},    
   pdfauthor={},     
   pdfsubject={},   
   pdfcreator={},   
   pdfproducer={}, 
   pdfkeywords={}, 
   pdfnewwindow=true,      
   colorlinks=true,       
   linkcolor=blue,          
   citecolor=blue,        
   filecolor=blue,      
   urlcolor=blue          
}

\newtheorem{theorem}{Theorem}[section]
\newtheorem{cor}[theorem]{Corollary}
\newtheorem{lem}[theorem]{Lemma}
\newtheorem{pro}[theorem]{Proposition}
\newtheorem{remark}[theorem]{Remark}
\newtheorem{Def}[theorem]{Definition}
\theoremstyle{definition}

\DeclareMathOperator{\Isom}{\mathsf{Isom}}%

\newcommand{\Bcal}{\mathcal{B}}

\newcommand{\Bd}{Beltrami differential}
\newcommand{\cts}{cotangent space}

\newcommand{\co}{curvature operator}
\newcommand{\ct}{curvature tensor}
\newcommand{\cs}{conformal structure}
\newcommand{\csi}{Cauchy-Schwarz inequality}

\newcommand{\Gf}{Green's function}
\newcommand{\hd}{H^{-1,1}(\D)}

\newcommand{\Hm}{Hilbert manifold}
\newcommand{\Hs}{Hilbert structure}
\newcommand{\hf}{holomorphic function}
\newcommand{\hqd}{holomorphic quadratic differential}

\newcommand{\hym}{hyperbolic metric}

\newcommand{\npd}{non-positive definite}
\newcommand{\ob}{orthonormal basis}
\newcommand{\Od}{\Omega^{-1,1}(\D)}

\newcommand{\Rc}{Ricci curvature}
\newcommand{\rc}{Riemannian curvature}
\newcommand{\rco}{Riemannian curvature operator}
\newcommand{\RS}{Riemann surface}
\newcommand{\sect}{sectional curvature}
\newcommand{\ts}{tangent space}

\newcommand{\TS}{Teichm\"{u}ller space}
\newcommand{\Tt}{Teichm\"{u}ller theory}

\newcommand{\thm}{twisted harmonic map}

\newcommand{\UTS}{universal Teichm\"{u}ller space}

\newcommand{\WP}{Weil-Petersson}
\newcommand{\wpm}{Weil-Petersson metric}
\newcommand{\wrt}{with respect to}

\newcommand{\be}{\begin{equation}}
\newcommand{\ene}{\end{equation}}
\newcommand{\br}{\begin{remark}}
\newcommand{\er}{\end{remark}}
\newcommand{\bl}{\begin{lem}}
\newcommand{\el}{\end{lem}}
\newcommand{\bcor}{\begin{cor}}
\newcommand{\ecor}{\end{cor}}
\newcommand{\bpro}{\begin{pro}}
\newcommand{\epro}{\end{pro}}
\newcommand{\ben}{\begin{enumerate}}
\newcommand{\een}{\end{enumerate}}
\newcommand{\bp}{\begin{proof}}
\newcommand{\ep}{\end{proof}}
\newcommand{\bpo}{\begin{pro}}
\newcommand{\epo}{\end{pro}}
\newcommand{\beq}{\begin{equation*}}
\newcommand{\eeq}{\end{equation*}}
\newcommand{\bear}{\begin{eqnarray}}
\newcommand{\eear}{\end{eqnarray}}
\newcommand{\beqar}{\begin{eqnarray*}}
\newcommand{\eeqar}{\end{eqnarray*}}
\newcommand{\bt}{\begin{theorem}}
\newcommand{\et}{\end{theorem}}

\newcommand{\C}{\mathbb{C}}
\newcommand{\R}{\mathbb{R}}
\newcommand{\J}{{\bf{J}}}

\newcommand{\Q}{\tilde{Q}}
\newcommand{\D}{\mathbb{D}}
\newcommand{\bi}{{\textbf{i}}}

\newcommand{\Acal}{\mathcal{A}}

\newcommand{\Tcal}{\mathcal{T}}

\newcommand{\ddl}[2]{\frac{d{#1}}{d{#2}}}

\newcommand{\ppl}[2]{\frac{\partial{#1}}{\partial{#2}}}

\numberwithin{equation}{section}

\allowdisplaybreaks

\def\XXint#1#2#3{{\setbox0=\hbox{$#1{#2#3}{\int}$}
    \vcenter{\hbox{$#2#3$}}\kern-.5\wd0}}

\makeatletter
\def\@citestyle{\m@th\upshape\mdseries}
\def\citeform#1{{\bfseries#1}}
\def\@cite#1#2{{%
  \@citestyle[\citeform{#1}\if@tempswa, #2\fi]}}
\@ifundefined{cite }{%
  \expandafter\let\csname cite \endcsname\cite
  \edef\cite{\@nx\protect\@xp\@nx\csname cite \endcsname}%
}{}
\makeatother

\begin{document}
\title[WEIL-PETERSSON Curvature operator]
{The Weil-Petersson curvature operator on the universal Teichm\"uller space}


\author{Zheng Huang}
\address[Z. ~H.]{Department of Mathematics, The City University of New York, Staten Island, NY 10314, USA}
\address{The Graduate Center, The City University of New York, 365 Fifth Ave., New York, NY 10016, USA}
\email{zheng.huang@csi.cuny.edu}

\author{Yunhui Wu}
\address[Y. ~W.]{ Yau Mathematical Sciences Center, Tsinghua University, Haidian District, Beijing 100084, China}
\email{yunhui\_wu@mail.tsinghua.edu.cn}





\begin{abstract} 
The {\UTS} is an infinitely dimensional generalization of the classical {\TS} of {\RS}s. It carries a natural {\Hs}, on 
which one can define a natural Riemannian metric, the {\wpm}. In this paper we investigate the {\WP} {\rco} 
$\tilde{Q}$ of the {\UTS} with the {\Hs}, and prove the following:
\ben
\item
$\tilde{Q}$ is {\npd}.
\item
$\tilde{Q}$ is a bounded operator.
\item
$\tilde{Q}$ is not compact; the set of the spectra of $\tilde{Q}$ is not discrete.
\een

As an application, we show that neither the Quaternionic hyperbolic space nor the Cayley plane can be totally geodesically immersed in 
the {\UTS} endowed with the {\WP} metric.
\end{abstract}


\maketitle

\section{Introduction}


\subsection{The {\WP} geometry on classical {\TS}} Moduli theory of {\RS}s and their generalizations continue 
to be inspiration for ideas and questions for many different mathematical fields since the times of 
Gauss and Riemann. In this paper, we study the {\WP} geometry of the {\UTS}.

Let $S_g$ be a closed oriented surface of genus $g$ where $g \ge 2$, and $\Tcal_g(S)$ be the {\TS} of 
$S_g$ (space of {\hym}s on $S_g$ modulo orientation preserving diffeomorphisms isotopic to the identity). 
The {\TS} $\Tcal_g(S)$ is a manifold of complex dimension $3g-3$, with its {\cts} at 
$(S,\sigma(z)|dz|^2) \in \Tcal_g(S)$ identified as the space of {\hqd}s $\phi(z)dz^2$ on the {\cs} of the 
{\hym} $\sigma(z)|dz|^2$. The {\wpm} on {\TS} is obtained by duality from the natural $L^2$ pairing of {\hqd}s. 
The {\WP} geometry of {\TS} has been extensively studied: it is a K\"{a}hlerian metric \cite{Ahl61}, 
incomplete \cite{Chu76, Wol75} yet geodesically convex \cite{Wol87}. Many features of the curvature 
property were also studied in detail by many authors (see a comprehensive survey \cite{Wlp11} and the 
book \cite{Wlp10}). Since intuitively we consider the {\UTS} contains {\TS}s of all genera, among those {\WP} 
curvature features; it is known that the Weil-Petersson metric has negative sectional curvature, with an explicit formula for 
the {\rc} tensor due to Tromba-Wolpert \cite{Tro86, Wlp86}, strongly negative curvature in the sense of Siu \cite{Schu86}, 
dual Nakano negative curvature \cite{LSY08}, various curvature bounds in terms of the 
genus \cite{Hua07b, Teo09, Wu16}, good behavior of the Riemannian {\co} on {\TS} \cite{Wu14, WW15}. One can also 
refer to \cite{BF06, Hua05, Hua07a, LSY04, LSYY13, Wlp11, Wlp10, Wlp12} for other aspects of the curvatures of the {\wpm}.      
\subsection{Main results} There are several well-known models of {\UTS}s. We will adapt the approach in \cite{TT06} and use the disk model to 
define the {\UTS} $T(1)$ as a quotient of the space of bounded {\Bd}s on the unit disk $\D$. Unlike the case in the 
classical {\TS}, the Petersson pairing for the bounded {\Bd}s on $\D$ is not well-defined on the whole {\ts} of the 
{\UTS} $T(1)$. To ramify this, Takhtajan-Teo \cite{TT06} defined a Hilbert structure on $T(1)$ such that the 
Petersson pairing is now meaningful on the {\ts} at any point in this Hilbert structure. We denote the {\UTS} 
with this Hilbert structure by $T_H(1)$. The resulting metric is the {\wpm} on $T_H(1)$. All terms will be 
defined rigorously in \S 2.

The Riemannian geometry of this infinitely dimensional deformation space $T_H(1)$ is very intriguing. Takhtajan-Teo 
showed the {\WP} metric on $T_H(1)$ has negative {\sect}, and constant {\Rc} \cite{TT06}, and Teo \cite{Teo09} proved 
the holomorphic {\sect} has no negative upper bound.

We are interested in the {\WP} {\co} on $T_H(1)$. In general there are some fundamental questions regarding linear operators on manifolds: 
whether the operator is signed, whether it is bounded, and the behavior of its eigenvalues. In this paper, we investigate the {\WP} {\co} 
along these question lines. In particular, we prove:

\bt\label{mt-1}
Let $\tilde{Q}$ be the {\WP} {\rco} on the {\UTS} $T_H(1)$, then 
\ben
\item
$\tilde{Q}$ is non-positive definite on $\wedge^2TT_H(1)$.
\item
For $C \in \wedge^2TT_H(1)$, $\Q(C, C) = 0$ if and only if there is an element $E \in \wedge^2TT_H(1)$ such that 
$C = E - \J\circ E$, where $\J\circ$ is defined above. 
\een
\et

As a direct corollary, we have:
\bcor\cite{TT06}
The {\sect} of the {\wpm} on $T_H(1)$ is negative.
\ecor

Our second result is:
\bt\label{mt-2}
The {\co} $\Q$ is bounded. More precisely, for any $V \in \wedge^2TT_H(1)$ with $\|V\|_{eu} =1$, we have 
$|\Q(V,V)|\le 16\sqrt{\frac{3}{\pi}}$, where $\|\cdot\|_{eu}$ is the Euclidean norm for the wedge product defined in \eqref{eu}. 
\et


A direct consequence of Theorem ~\ref{mt-2} is:
\bcor\cite{GBR15}
The Riemannian {\WP} curvature tensor (defined in \eqref{ct3}) is bounded.
\ecor

Being bounded and non-positively definite are properties for the {\WP} {\co} on certain part of the classical {\TS} as well 
\cite{Wu14, WW15}, but noncompactness of $\tilde{Q}$ is a more distinctive feature for $T_H(1)$. Our next result is:
\bt\label{mt-3}
The {\co} $\Q$ is not a compact operator, more specifically, the set of spectra of $\Q$ is not discrete on the interval $[-16\sqrt{\frac{3}{\pi}},0)$.
\et

As an important application, in the last part of this paper we will address some rigidity questions on harmonic maps from certain symmetric 
spaces into $T_H(1)$. For harmonic map from a domain, which is either the Quaternionic hyperbolic space or the Cayley plane, into a 
non-positive curved target space, many beautiful rigidity results were established in \cite{DM15, GS92, JY97, MSY93} and others. We prove the following: 
\begin{theorem}\label{mt-4}
Let $\Gamma$ be a lattice in a semisimple Lie group $G$ which is either $Sp(m,1)$ or $F_{4}^{-20}$, and let $\Isom(T_H(1))$ be the isometry 
group of $T_H(1)$ {\wrt} the {\wpm}. Then, any {\thm} $f$ from $G/\Gamma$ into $T_H(1)$ must be a constant, {\wrt} each homomorphism 
$\rho: \Gamma \rightarrow \Isom(T_H(1))$. Here the twisted map $f$ {\wrt} $\rho$ means that 
$f(\gamma \circ Y)=\rho (\gamma)\circ f(Y)$, for all $\gamma \in \Gamma$.
\end{theorem}
\subsection{Methods in the proofs}

An immediate difficulty we have to cope with is that $T_H(1)$ is an infinite dimensional manifold. There is 
however a basis for tangent vectors for the Hilbert structure that we can work with. With this basis, the {\WP} 
{\rc} tensor takes an explicit form. To prove the first two results, we need to generalize techniques developed in \cite{Wu14, WW15} 
carefully and rigorously to the case of infinite dimensional Hilbert spaces.

Proof of the Theorem ~\ref{mt-3} is different. We prove a key estimate for the operator on an $n$-dimensional subspace 
(Proposition ~\ref{4-4}), then bound the spectra of the {\co} by the corresponding spectra of its projection onto this subspace to derive a contradiction.

\subsection{Plan of the paper}
The organization of the paper is as follows: in \S 2, we set up notations and preliminaries, in particular, we 
restrict ourselves in the classical setting to define {\TS} of closed surfaces and the {\wpm} in \S2.1, its {\co} on 
{\TS} is set up in \S2.2, then we define the {\UTS} and its Hilbert structure in \S 2.3, and introduce the basis 
for tangent vectors for the $T_H(1)$, and describe the {\WP} {\rco} on the {\UTS} in \S2.4. Main theorems are proved in sections \S3, \S 4 and \S 5. And in 
the last section \S 6 we prove Theorem ~\ref{mt-4}.

\subsection{Acknowledgment} We acknowledge supports from U.S. national science foundation 
grants DMS 1107452, 1107263, 1107367 ``RNMS: Geometric Structures and Representation 
varieties" (the GEAR Network). This work was supported by a grant from the Simons Foundation 
(\#359635, Zheng Huang) and a research award from the PSC-CUNY. Part of the work is completed when the second named author was a G. C. Evans 
Instructor at Rice University, he would like to thanks to the mathematics department for their support. He would also like to acknowledge a 
start-up research fund from Tsinghua University to finish this work. The authors would like to thank an anonymous referee whose comments are very helpful to impove the paper.
\section{Preliminaries}
\subsection{{\TS} and its {\wpm}}
Let $\D$ be the unit disk with the Poincar\'e metric, and $S$ be a closed oriented surface of genus $g > 1$. 
Then by the uniformization theorem we have a hyperbolic structure $X = \D\backslash\Gamma$ on $S$, where 
$\Gamma \subset PSL(2,\R)$ is a Fuchsian group, and $PSL(2,\R)$ is the group of orientation preserving 
isometries of $\D$. Writing $\{z\}$ as the complex coordinate on $\D$, the Poincar\'e metric is explicitly given as  
\beq
\rho(z) = \frac{4}{(1-|z|^2)^2}dzd\bar{z}.
\eeq
It descends to a {\hym} on the {\RS} $X=\D\backslash\Gamma$, which we denote by $\sigma(z)|dz|^2$. Spaces 
of {\Bd}s and {\hqd}s on {\RS}s play a fundamental role in {\Tt}, and let us describe these spaces.
\ben
\item
$\Acal^{-1,1}(X)$: the space of bounded {\Bd}s on $X=\D\backslash\Gamma$. A {\Bd} on a {\RS} is a $(-1,1)$ 
form in the form of $\mu(z)\ddl{\bar{z}}{z}$, where $\mu(z)$ is a function on $\D$ satisfying:
\beq
\mu(\gamma(z))\frac{\overline{\gamma'(z)}}{\gamma'(z)} = \mu(z), \ \ \ \forall \gamma \in \Gamma.
\eeq
\item
$\Bcal^{-1,1}(X)$: the unit ball of $\Acal^{-1,1}(X)$, namely,
\beq
\Bcal^{-1,1}(X) = \{\mu(z)\ddl{\bar{z}}{z}\in \Acal^{-1,1}(X): \|\mu\|_{\infty} = sup_{z\in\D}|\mu(z)|<1 \}.
\eeq
\item
$Q(X)$: the space of {\hqd}s on $X$. A {\hqd} is a $(2,0)$ form taking the form $q(z)dz^2$, where $q(z)$ is 
a holomorphic function on $\D$ satisfying:
 \beq
q(\gamma(z))[\gamma'(z)]^2 = q(z), \ \ \ \forall \gamma \in \Gamma.
\eeq
It is a basic fact in {\RS} theory that $Q(X)$ is a Banach space of real dimension $6g-6$.
\item
$\Omega^{-1,1}(X)$: the space of {\it harmonic {\Bd}s} on $X$. A {\Bd} $\nu(z)\ddl{\bar{z}}{z} \in \Acal^{-1,1}(X)$ 
is harmonic if there is a {\hqd} $q(z)dz^2 \in Q(X)$ such that 
 \be\label{nu}
\nu(z)\ddl{\bar{z}}{z} = \frac{\overline{q(z)dz^2}}{\sigma(z)|dz|^2},
\ene
where $\sigma(z)|dz|^2$ is the {\hym} on $X$. Seeing from $\D$, the space $\Omega^{-1,1}(X)$ consists of 
functions 
\be\label{nu2}
\nu(z) = \frac{(1-|z|^2)^2}{4}\overline{q(z)}.
\ene
\een
The {\TS} $\Tcal_g(S)$ is the space of {\hym}s on the surface $S$, modulo orientation preserving 
biholomorphisms. Real analytically $\Tcal_g(S)$ is isomorphic to $\Bcal^{-1,1}(X)\backslash\sim$, where two 
{\Bd}s are equivalent if the unique quasiconformal maps between the extended complex plane coincide on 
the unit circle. At each point $X \in \Tcal_g(S)$, its tangent space is identified as the space $\Omega^{-1,1}(X)$, 
while the {\cts} at $X$ is identified as the space $Q(X)$.

Given two tangent vectors $\mu(z)\ddl{\bar{z}}{z}$ and $\nu(z)\ddl{\bar{z}}{z}$ in $\Omega^{-1,1}(X)$, the {\wpm} 
is defined as the following (Petersson) pairing:
\be\label{wpm}
\langle\mu,\nu\rangle_{WP} = \int_{X=\D\backslash\Gamma}\mu\bar\nu dA,
\ene
where $dA = \sigma|dz|^2$ is the hyperbolic area element on $X$. Writing as a metric tensor, we have 
\beq
g_{i\bar{j}} =  \int_{X}\mu_i\bar{\nu}_j dA,
\eeq
This is a Riemannian metric with many nice properties. There is an explicit formula for its {\ct} due to 
Tromba-Wolpert (\cite{Tro86,Wlp86}): 
\be\label{ct}
R_{i\bar{j}k\bar{\ell}} = \int_X D(\mu_i\bar{\mu}_j)(\mu_k\bar{\mu}_\ell)dA + 
\int_X D(\mu_i\bar{\mu}_\ell)(\mu_k\bar{\mu}_j)dA .
\ene
Here the operator $D$ is defined as 
\be\label{D}
D = -2(\Delta - 2)^{-1},
\ene
where $\Delta = \frac{-4}{\sigma(z)}\partial_z\partial_{\bar{z}}$ is the Laplace operator on $X$ {\wrt} the {\hym} $\sigma(z)dA$. This operator $D$ is 
fundamental in {\Tt}, and the following is well-known (see for instance \cite{Wlp86}):
\bpro\label{D2}
The operator $D = -2(\Delta - 2)^{-1}$ is a positive, self-adjoint operator on $C^{\infty}(X)$. Furthermore, 
let $G(w,z)$ be a Green's function for $D$, then $G(w,z)$ is positive, and $G(w,z) = G(z,w)$: 
$\forall f \in C^\infty(X)$, we have 
\be\label{green}
D(f)(z) = \int_{w\in X}G(z,w)f(w)dA(w).
\ene
\epro
To simplify our calculations, we introduce the following notation:
\begin{Def}
For any element $\mu$'s in the {\ts} $\Omega^{-1,1}(X)$, we set:
\be\label{sim}
(i\bar{j},k\bar{\ell}) =\int_X D(\mu_i\bar{\mu}_j)(\mu_k\bar{\mu}_\ell)dA.
\ene
Using this notation, the {\WP} {\ct} formula on {\TS} becomes
\be\label{ct2}
R_{i\bar{j}k\bar{\ell}} = (i\bar{j},k\bar{\ell}) + (i\bar{\ell},k\bar{j}).
\ene
\end{Def}
\subsection{The {\WP} {\co} on {\TS}}
We now introduce the {\rco} for the {\wpm} on {\TS} $\Tcal_g(S)$. Note that this is a matrix of the real order 
$(6g-6)^2\times (6g-6)^2$, whose diagonal entries are the {\sect}s.

Let $U$ be a neighborhood of $X$ in {\TS} $\Tcal_g(S)$, and we have $\{t_{1},t_{2},...,t_{3g-3}\}$ as a local 
holomorphic coordinate on $U$, where $t_{i}=x_{i}+\textbf{i}y_{i} (1\leq i \leq 3g-3)$. Then 
$\{x_{1},x_{2},...,x_{3g-3},y_{1},y_{2},...,y_{3g-3}\}$ forms a real smooth coordinate in $U$, and 
\beq
\ppl{}{x_i} =\ppl{}{t_i}+\ppl{}{\bar{t}_i}, \ \ \ \ppl{}{y_i} =\textbf{i}(\ppl{}{t_i}-\ppl{}{\bar{t}_i}).
\eeq

Let $T\Tcal_g(S)$ be the real tangent bundle of $\Tcal_g(S)$ and $\wedge^{2}T\Tcal_g(S)$ be the exterior 
wedge product of $T\Tcal_g(S)$ and itself. For any $X \in U$, we have
\beq
T_{X}\Tcal_g(S)=Span\{\ppl{}{x_i}(X),\ppl{}{y_i}(X)\}_{1\leq i,j\leq3g-3},
\eeq
and
\be\label{wedge}
\wedge^{2}T_X\Tcal_g(S)= Span\{\ppl{}{x_i}\wedge \ppl{}{x_j}, \ppl{}{x_k}\wedge \ppl{}{y_{\ell}},
\ppl{}{y_m}\wedge \ppl{}{y_n}\}.
\ene
\begin{Def}\label{co}
The {\WP} {\co} $\Q$ on {\TS} is defined on $\wedge^{2}T\Tcal_g(S)$ by 
\beq
\Q(V_1\wedge V_2, V_3\wedge V_4) = R(V_1, V_2, V_3, V_4),
\eeq
where $V$'s are tangent vectors at $X$, and $R$ is the {\ct}.
\end{Def}
If we take a real {\ob} $\{e_i\}_{i = 1, 2, \cdots, 6g-6}$ for $T_X\Tcal_g(S)$, and set 
$R_{ijk\ell}=\langle R(e_i,e_j)e_k,e_\ell\rangle$, then 
\beq
\wedge^{2}T_X\Tcal_g(S)= Span\{e_i\wedge e_j\}_{1\le i<j\le (6g-6)},
\eeq
and the {\co} $\Q: \wedge^{2}T_X\Tcal_g(S)\rightarrow\wedge^{2}T_X\Tcal_g(S)$, for real coefficients $a_{ij}$, 
can be expressed as follows:
\be\label{Q1}
\Q(\sum\limits_{1\le i<j\le (6g-6)}a_{ij}e_i\wedge e_j) = 
\sum\limits_{1\le i< j\le (6g-6)}\sum\limits_{1\le k<\ell\le (6g-6)}a_{ij}R_{ijk\ell}e_k\wedge e_\ell.
\ene
In \cite{Wu14} the second named author proved the {\co} $\Q$ is non-positively definite on {\TS}. 
Further analysis on $\Q$ was studied in \cite{WW15}. We will generalize 
this fundamental operator to the setting of the {\UTS} and reveal some geometric features for the {\wpm} on the 
{\UTS}. 
\subsection{The {\UTS} and its Hilbert structure}
Introduced by Bers (\cite{Ber65}), the {\UTS} $T(1)$ is a central subject for the theory of univalent functions. 
It contains all {\TS}s $\Tcal_g(S)$ of closed surfaces which are complex submanifolds.

Recall that every {\RS} (or hyperbolic structure) $X$ on a closed surface $S$ is quotient of the Poincar\'e disk 
with a Fuchsian group $\Gamma$: $X = \D\backslash\Gamma$. Previously in \S 2.1, we have {\TS} 
$\Tcal_g(S)$ isomorphic to a quotient space $\Bcal^{-1,1}(X)\backslash\sim$, where $\Bcal^{-1,1}(X)$ is the space 
of bounded {\Bd}s on $X$ with super-norm less than one, and two such {\Bd}s are equivalent if the unique 
quasiconformal maps induced by them between the extended complex plane coincide on the unit circle.

Let us set up some notations before we proceed. Letting $\Gamma$ be the identity group, we work in the Poincar\'e 
disk $\D$, we have similarly with \S2.1:
\ben
\item
$\Acal^{-1,1}(\D)$: the space of bounded functions on $\D$. 
\item
$\Bcal^{-1,1}(\D)$: the unit ball of $\Acal^{-1,1}(\D)$, namely,
\beq
\Bcal^{-1,1}(\D) = \{\mu(z)\in \Acal^{-1,1}(D): \|\mu\|_{\infty} = sup_{z\in\D}|\mu(z)|<1 \}.
\eeq
\item
We will need two spaces of holomorphic functions on $\D$, both are analog to the space $Q(X)$, the space 
of {\hqd}s on $X$. Let us define
\be\label{Ainf}
A_\infty(\D) = \{q(z): \bar\partial q =0, \|q\|_{\infty} = sup_{z\in\D}\frac{|q(z)|}{\rho(z)}<\infty \},
\ene
where $\rho(z) = \frac{4}{(1-|z|^2)^2}dzd\bar{z}$ is the {\hym} on $\D$. This is the space of {\hf}s on $\D$ with 
finite super-norm defined within \eqref{Ainf}.

We also define 
\be\label{A2}
A_2(\D) = \{q(z): \bar\partial q =0, \|q\|_2^2 = \int_\D \frac{|q(z)|^2}{\rho(z)}|dz|^2<\infty \}.
\ene
This is the space of {\hf}s on $\D$ with finite $L^2$-norm defined within \eqref{A2}.
\item
For the notion of generalized ``harmonic {\Bd}s" on $\D$, we also have two spaces to introduce:
\be\label{OmegaD}
\Od= \{\nu(z)\in \Acal^{-1,1}(\D): \nu = \frac{\bar{q}}{\rho(z)}\  \text{for some}\  q \in A_\infty(\D)\},
\ene
and 
\be\label{HD}
\hd= \{\nu(z)\in \Acal^{-1,1}(\D): \nu = \frac{\bar{q}}{\rho(z)}\  \text{for some}\  q \in A_2(\D)\}.
\ene
\een
\begin{Def}
The {\UTS} $T(1) = \Bcal^{-1,1}(\D)\backslash\sim$, where $\mu \sim \nu \in \Bcal^{-1,1}(\D)$ if and only if 
$w_\mu = w_\nu$ on the unit circle, and $w_\mu$ is the unique quasiconformal map between extended complex 
planes which fixes the points $-1,-i,1$, and solves the Beltrami equation $f_{\bar{z}} = \mu f_z$. 
\end{Def}
At any point in the {\UTS} $T(1)$, the {\cts} is naturally identified with the Banach space $A_\infty(\D)$ 
defined in \eqref{Ainf}, and the {\ts} is identified with the space $\Od$ defined in \eqref{OmegaD}. It 
is then clear the Petersson pairing of functions in the space $\Od$ is not well-defined. However, for 
any $\mu, \nu \in \hd\subset \Od$, we write the Petersson pairing as the following inner 
product: 
\be\label{pair}
\langle\mu,\nu\rangle = \int_\D\mu\bar{\nu}\rho(z)|dz|^2.
\ene
Then this defines a Hilbert structure on the {\UTS} $T(1)$, introduced in (\cite{TT06}), namely, $T(1)$ 
endowed with this inner product, becomes an infinite dimensional complex manifold and Hilbert space. We denote 
this Hilbert manifold $T_H(1)$, which consists of all the points of the {\UTS} $T(1)$, with tangent space identified 
as $\hd$, a sub-Hilbert space of the Banach space $\Od$. We call the resulting metric from \eqref{pair} the 
{\it {\wpm}} on $T_H(1)$. The space we are dealing with is still very complicated: in the corresponding topology 
induced from the inner product above, the {\Hm} $T_H(1)$ is a disjoint union of uncountably many components (\cite{TT06}).

One of the most important tools for us is the {\Gf} for the operator $D$ on the disk. We abuse our notation to denote the operator 
$D = -2(\Delta_\rho-2)^{-1}$ and $G(z,w)$ its {\Gf}, where $\Delta_\rho$ is the Laplace operator on the 
Poincar\'e disk $\D$. Let us organize some properties we will use later into the following proposition.
\bpro\cite{Hej76}\label{greenpro}
The {\Gf} $G(z,w) $ satisfies the following properties:
\ben
\item Positivity: $G(z,w) > 0$ for all $z, w \in \D$;
\item Symmetry: $G(z,w) = G(w,z)$ for all $z, w \in \D$;
\item Unit hyperbolic area: $\int_{\D}G(z,w)dA(w) = 1$ for all $z \in \D$;
\item We denote $BC^{\infty}(\D)$ the space of bounded smooth functions on $\D$, then for $\forall f(z) \in BC^{\infty}(\D)$, 
\be\label{green2}
D(f)(z) = \int_{w\in \D}G(z,w)f(w)dA(w). 
\ene
Moreover, $D(f) \in BC^{\infty}(\D)$.
\een
\epo
\subsection{The {\co} on the {\UTS}} We have defined the {\Hm} $T_H(1)$ and its Riemannian 
metric \eqref{pair} for its {\ts} $\hd$ which we will work with for the rest of the paper, let us now generalize the 
concept of the {\co} (Definition ~\ref{co}) for {\TS} to $T_H(1)$. This has been done in more abstract settings, 
see for instance [pages 238-239, \cite{Lan99}] or \cite{Duc13}.

We work in the Poincar\'e disk $\D$. On one hand, without Fuchsian group action, we are forced to deal with an 
infinite dimensional space of certain functions, on the other hand, the {\hym} is explicit. This leads to some 
explicit calculations that one can take advantage of. First we note that the {\ts} $\hd$ has an explicit {\ob}: we set, 
$n \ge 2$,  
\be\label{ob}
\mu_{n-1} = \frac{(1-|z|^2)^2}{4}\sqrt{\frac{2n^3-2n}{\pi}}{\bar{z}}^{n-2}.
\ene
\bl\label{teobasis}\cite{TT06,Teo09}
The set $\{\mu_i\}_{i \ge 1}$ forms an {\ob} {\wrt} the {\wpm} on $H^{-1,1}(\D)$.
\el
Moreover, Takhtajan-Teo established the {\ct} formula for the {\wpm} on $T_H(1)$, which takes the same form 
as Tromba-Wolpert's formula for {\TS} of closed surfaces:
\bt\cite{TT06}
For $\{\mu\}$'s in $H^{-1,1}(\D)$, the Riemannian {\ct} for the {\wpm} \eqref{pair} is given by: 
\be\label{ct3}
R_{i\bar{j}k\bar{\ell}} = \int_\D D(\mu_i\bar{\mu}_j)(\mu_k\bar{\mu}_\ell)dA + 
\int_\D D(\mu_i\bar{\mu}_\ell)(\mu_k\bar{\mu}_j)dA. 
\ene
Here we abuse our notation to use $D=-2(\Delta_\rho-2)^{-1}$, where $\Delta_\rho$ is the Laplace 
operator on the Poincar\'e disk $\D$. 
\et
 
Let $U$ be a neighborhood of $p \in T_H(1)$ and $\{t_{1},t_{2},\cdots\}$ be a local holomorphic 
coordinate system on $U$ such that $\{t_{i}(p)=\mu_{i}\}_{i \geq 1}$ is orthonormal at $p$, where $\mu_i$'s 
are explicitly defined in \eqref{ob}, we write $t_{i}=x_{i}+\textbf{i}y_{i} \ \ (i \geq 1)$, then 
$\{x_{1},y_{1},x_2,y_2, \cdots\}$ is a real smooth coordinate system in $U$, and we have:
\beq
\ppl{}{x_i} =\ppl{}{t_i}+\ppl{}{\bar{t}_i}, \ \ \ \ppl{}{y_i} =\textbf{i}(\ppl{}{t_i}-\ppl{}{\bar{t}_i}).
\eeq
Let $TT_H(1)$ be the real tangent bundle of $T_H(1)$ and $\wedge^{2}TT_H(1)$ be the exterior wedge 
product of $TT_H(1)$ and itself. For any $p \in U$, we have
\beq
T_{p}T_H(1)=Span\{\ppl{}{x_i}(p),\ppl{}{y_j}(p)\}_{ i.j\geq 1},
\eeq
and
\beq
\wedge^{2}TT_H(1)=Span\{\ppl{}{x_i}\wedge\ppl{}{x_j}, \ppl{}{x_k}\wedge \ppl{}{y_\ell}, 
\ppl{}{y_m}\wedge \ppl{}{y_n}\}.
\eeq
Following Lang [Chapter 9, \cite{Lan99}], we define:
\begin{Def}
The {\WP} {\co} $\Q:\wedge^{2}TT_H(1)\rightarrow \wedge^{2}TT_H(1)$ is given as
\beq 
\Q(V_1\wedge V_2,V_3\wedge V_4)=R(V_1,V_2,V_3,V_4),
\eeq 
and extended linearly, where $V_{i}$ are real tangent vectors, and $R$ is the {\ct} for the {\wpm}.
\end{Def}
It is easy to see that $\Q$ is a bilinear symmetric form.
\section{Non-Positive Definiteness and Zero Level Set}
In this section, we prove the first part of Theorem ~\ref{mt-1}:
\bt\label{npd}
The operator $\Q$ is {\npd}.
\et
The strategy of our proof is the most direct approach, namely, lengthy but careful calculations using explicit nature of 
both the {\hym} on $\D$, and the {\ob} on $\hd$ given by \eqref{ob}. We will verify the theorem by 
calculating with various combinations of bases elements, then extend bilinearly. We follow closely the argument 
in the proof of Theorem 1.1 in \cite{Wu14}, which was inspired by calculations in \cite{LSY08}. 
\subsection{Preparation}
Note that the version of the operator $D = -2(\Delta-2)^{-1}$ on a closed surface is positive and self-adjoint 
on $L^2(X=\D\backslash\Gamma)$, and it plays a fundamental role in {\Tt}, but in the case of $\D$, the operator $D = -2(\Delta_\rho-2)^{-1}$ is 
noncompact, therefore we have to justify several properties carefully. 
\bpro\label{D1}
We have the following:
\ben
\item
The operator $D = -2(\Delta_\rho-2)^{-1}$ is self-adjoint on $L^2(\D)\cap BC^{\infty}(\D)$;
\item
For any $f \in L^2(\D)\cap BC^{\infty}(\D)$, we have also $D(f) \in L^2(\D)\cap BC^{\infty}(\D)$;
\item
The operator $D = -2(\Delta_\rho-2)^{-1}$ is positive on $L^2(\D)\cap BC^{\infty}(\D)$.
\een
\epro
\bp 
(i). For all $f, h \in L^2(\D)\cap BC^{\infty}(\D)$, we have 
\beqar
\int_\D D(f)h dA(z) &=& \int_\D \int_\D G(z,w)f(w)dA(w) h(z)dA(z) \\
&=& \int_\D \int_\D G(z,w)h(z)dA(z) f(w)dA(w) \\
&=& \int_\D fD(h) dA(w).
\eeqar
(ii). From Proposition ~\ref{greenpro}, we know $D(f), D(f^2)\in BC^{\infty}(\D)$. Using the positivity of the {\Gf} $G(z,w)$, 
and $\int_\D G(z,w)dA(w) = 1$, we estimate with the {\csi}:
\beqar
\int_\D |D(f(z))|^2 dA(z)  &=& \int_\D \{\int_\D|G(z,w)f(w)| dA(w)\}^2dA(z) \\
&\le& \int_\D(\int_\D |G(z,w)f^2(w)| dA(w))\bullet (1)dA(z)\\
&=& \int_D D(|f^2|)dA(z) \\
&=& \int_D |f|^2dA(z).
\eeqar
The last step we used self-adjointness of $D$ and the fact that $D(1) = 1$. This proves $D(f) \in L^2(\D)$.

(iii) Given any real function $f \in L^2(\D)\cap BC^{\infty}(\D)$, let us denote $u = D(f)$ and by (ii) above, it also lies in 
$L^2(\D)\cap BC^{\infty}(\D)$, then $f = -\frac12(\Delta_\rho -2)u$, and 
\beqar
\int_\D D(f)f dA &=& -\frac12\int_\D u(\Delta_\rho -2)u dA \\
&=& \int_\D u^2 dA - \frac12\int_\D u\Delta_\rho udA \\
&\ge &\int_\D u^2 dA \ge 0.
\eeqar
Here we used that $\Delta_\rho$ is negative definite on $L^2(\D)\cap BC^{\infty}(\D)$ (\cite{Hej76}).

The case when $f$ is complex valued can be proved similarly after working on real and imaginary parts separately.
\ep
Recall that $\{\mu_1, \mu_2, \cdots\}$ forms an {\ob} for $\hd$, where $\mu_i$'s are given explicitly in 
\eqref{ob}. Using the coordinate system described in \S2.4, we have 
\beq
\wedge^{2}TT_H(1)=Span\{\ppl{}{x_i}\wedge\ppl{}{x_j}, \ppl{}{x_k}\wedge \ppl{}{y_\ell}, 
\ppl{}{y_m}\wedge \ppl{}{y_n}\}.
\eeq
Naturally we will work with these three combinations. Let us define a few terms to simplify our calculations:
\ben
\item
Consider $\sum\limits_{i j}{a_{ij}\ppl{}{x_i}\wedge \ppl{}{x_j}}$, where $a_{ij}$ are real. We denote
\be\label{F}
F(z,w)=\sum_{i,j\ge1}{a_{ij}\mu_{i}(w)\cdot \overline{\mu_{j}(z)}}.
\ene
\item
The Green's function of the operator $D$: $G(z,w) = G(w,z)$.
\item
Consider $\sum\limits_{i j}{b_{ij}\ppl{}{x_i}\wedge \ppl{}{y_j}}$, where $b_{ij}$ are real. We denote
\be\label{H}
H(z,w)=\sum_{i,j\ge1}{b_{ij}\mu_{i}(w)\cdot \overline{\mu_{j}(z)}}.
\ene
\een
There are three types of basis elements in $\wedge^{2}TT_H(1)$, however, in terms of the {\co}, we only 
have to work with the first two types because of the next lemma:

\bl\label{reduction}
We have the following:

 \ben
\item $\Q(\ppl{}{x_i}\wedge\ppl{}{x_j}, \ppl{}{y_k}\wedge \ppl{}{y_\ell}) = \Q(\ppl{}{x_i}\wedge\ppl{}{x_j}, \ppl{}{x_k}\wedge \ppl{}{x_\ell}).$ \\

\item $\Q(\ppl{}{x_i}\wedge\ppl{}{y_j}, \ppl{}{y_k}\wedge \ppl{}{y_\ell})= \Q(\ppl{}{x_i}\wedge\ppl{}{y_j}, \ppl{}{x_k}\wedge \ppl{}{x_\ell}).$ 
  \een

\el
\bp
The {\wpm} on the {\Hm} $T_H(1)$ is K\"ahler-Einstein (\cite{TT06}), therefore its associated complex 
structure $\J$ is an isometry on the {\ts} $\hd$ such that ${\J} \ppl{}{x_i} = \ppl{}{y_i}$ and ${\J}^2 = -\text{id}$. 
Now it is easy to verify: 
\beqar
\Q(\ppl{}{x_i}\wedge\ppl{}{x_j}, \ppl{}{y_k}\wedge \ppl{}{y_\ell}) &=&
R(\ppl{}{x_i},\ppl{}{x_j}, \J\ppl{}{x_k},\J\ppl{}{x_\ell})\\
&=& R(\ppl{}{x_i},\ppl{}{x_j}, \ppl{}{x_k},\ppl{}{x_\ell}) \\&=&
\Q(\ppl{}{x_i}\wedge\ppl{}{x_j}, \ppl{}{x_k}\wedge \ppl{}{x_\ell}).
\eeqar
The other equality is proved similarly.
\ep
\subsection{Proof of Theorem ~\ref{npd}} The curvature tensor in equation \eqref{ct3} has two terms. Set
\[(ij,kl):=\int_{\D}D(\mu_i \overline{\mu}_j)(\mu_k \overline{\mu}_l)dA.\]
Thus, the curvature tensor satisfies
\be
R_{i\bar{j}k\bar{\ell}}=(ij,kl)+(il,kj).
\ene
\quad

\bp [Proof of Theorem ~\ref{npd}]
We write 
\be\label{AB}
A = \sum\limits_{i j}{a_{ij}\ppl{}{x_i}\wedge \ppl{}{x_j}}, \ \ 
B=\sum\limits_{i j}{b_{ij}\ppl{}{x_i}\wedge \ppl{}{y_j}}.
\ene
By Lemma ~\ref{reduction}, we only have to show $\Q(A+B, A+B) \le 0$. We pause to give an 
expression for $\Q(A+B, A+B)$. Since $\Q (A, B) = \Q(B,A)$, we have
\beq
\Q(A+B,A+B) = \Q(A,A) + 2\Q(A,B)+ \Q(B,B).
\eeq
Now we work with these terms.

\bl \label{3.4-1}
Using above notations, we have 
\beqar
\Q(A,A)&=& -4\int_{\D}D(\Im F(z,z))(\Im F(z,z))dA(z) \\
&\ \ \ \ +& 2\Re\{\iint_{\D\times\D}G(w,z)F(z,w)F(w,z)dA(w)dA(z)\} \\
&\ \ \ \ -&2\iint_{\D\times\D}G(w,z)|F(z,w)|^2dA(w)dA(z).
\eeqar
\el
\bp [Proof of Lemma \ref{3.4-1}]
First we recall the {\WP} {\ct} formula \eqref{ct3}, notation in \eqref{sim}, 
and $\ppl{}{x_i} = \ppl{}{t_i} + \ppl{}{\bar{t}_i}$, then we take advantage of the Green's function $G(z,w)$ 
for $D$ and the fact that $D$ is self-adjoint on $L^2(\D)\cap BC^{\infty}(\D)$ to calculate as follows:
\beqar
\Q(A,A)&=&\sum_{i,j,k,\ell}a_{ij}a_{k\ell}(R_{i\bar{j}k\bar{\ell}}+R_{i\bar{j}\bar{k}\ell}+
R_{\bar{i}jk\bar{\ell}}+R_{\bar{i}j\bar{k}\ell})\\
&=& \sum_{i,j,k,\ell}a_{ij}a_{k\ell}((i\bar{j},k\bar{\ell})+(i\bar{\ell},k\bar{j})-(i\bar{j},\ell\bar{k})-(i\bar{k},\ell\bar{j})\\
&\ \ \ \ -&(j\bar{i},k\bar{\ell})-(j\bar{\ell},k\bar{i})+(j\bar{i},\ell\bar{k})+(j\bar{k},\ell\bar{i})) \\
&=& \sum_{i,j,k,\ell}a_{ij}a_{k\ell}(i\bar{j}-j\bar{i},k\bar{\ell}-\ell\bar{k})+ 
\sum_{i,j,k,\ell}a_{ij}a_{k\ell}((i\bar{\ell},k\bar{j})+(\ell\bar{i},j\bar{k}))\\
&\ \ \ \ -& \sum_{i,j,k,\ell}a_{ij}a_{k\ell}((i\bar{k},\ell\bar{j})+(j\bar{\ell},k\bar{i}))\\
&=& \int_{\D}D(F(z,z)-\overline{F(z,z)})(F(z,z) - \overline{F(z,z)})dA(z) \\
&\ \ \ \ +& 2 \Re\{\sum_{i,j,k,\ell}a_{ij}a_{k\ell}(i\bar{\ell},k\bar{j})\} -2\Re\{\sum_{i,j,k,\ell}a_{ij}a_{k\ell}(i\bar{k},\ell\bar{j})\}\\
&=& -4\int_{\D}D(\Im F(z,z))(\Im F(z,z))dA(z) \\
&\ \ \ \ +& 2 \Re\{\sum_{i,j,k,\ell}a_{ij}a_{k\ell}(i\bar{\ell},k\bar{j})\} -2\Re\{\sum_{i,j,k,\ell}a_{ij}a_{k\ell}(i\bar{k},\ell\bar{j})\}.
\eeqar

For the second term in the equation above,
\begin{small}
\begin{eqnarray*}
&& \Re\{\sum_{i,\ell}a_{ij}a_{k\ell}((i\overline{\ell},k\overline{j})\}\\
&=& \Re\{\int_{\D} D(\sum_{i}a_{ij}\mu_{i}\overline{\mu_{\ell}})(\sum_{k}a_{k\ell}\mu_{k}\overline{\mu_{j}})dA(z)\}\\
&=&  \Re\{\iint_{\D \times \D}G(w,z)\sum_{i}a_{ij}\mu_{i}(w)\overline{\mu_{\ell}(w)}(\sum_{k}a_{k\ell}\mu_{k}(z)\overline{\mu_{j}}(z))dA(z)dA(w)\}\\
&=& \Re\{\int_{\D\times \D}G(z,w)F(z,w)F(w,z)dA(w)dA(z)\}.
\end{eqnarray*}
\end{small}

Similarly, we have 
\begin{small}
\begin{eqnarray*}
 \Re\{\sum_{i,\ell}a_{ij}a_{k\ell}((i\overline{k},\ell\overline{j})\}=\iint_{\D\times\D}G(w,z)|F(z,w)|^2dA(w)dA(z).
\end{eqnarray*}
\end{small}
Then, the lemma follows by the equations above.
\ep

\bl \label{3.4-2}
Using above notations, we have 
\beqar
\Q(B,B)&=& -4\int_{\D}D(\Re H(z,z))(\Re H(z,z))dA(z) \\
&\ \ \ \ -&2\Re\{\iint_{\D\times\D}G(w,z)H(z,w)H(w,z)dA(w)dA(z)\} \\
&\ \ \ \ -&2\iint_{\D\times\D}G(w,z)|H(z,w)|^2dA(w)dA(z).
\eeqar
\el
\bp [Proof of Lemma \ref{3.4-2}]
Using $\ppl{}{y_i} = \bi(\ppl{}{t_i}- \ppl{}{\bar{t}_i})$, we have 
\beqar
&&\Q(B,B) =\Q(\sum_{ij}b_{ij}\frac{\partial}{\partial x_{i}}\wedge \frac{\partial}{\partial y_{j}},\sum_{ij}b_{ij}\frac{\partial}{\partial x_{i}}\wedge \frac{\partial}{\partial y_{j}})\\
&=& -\sum_{i,j,k,\ell}b_{ij}b_{k\ell}(R_{i\overline{j}k\overline{\ell}}-R_{i\overline{j}\overline{k}\ell}-R_{\overline{i}jk\overline{\ell}}+R_{\overline{i}j\overline{k}\ell})\\
&=& -\sum_{i,j,k,\ell}b_{ij}b_{k\ell}(R_{i\overline{j}k\overline{\ell}}+R_{i\overline{j}\ell\overline{k}}+R_{j\overline{i}k\overline{\ell}}+R_{j\overline{i}\ell\overline{k}})\\
&=& -\sum_{i,j,k,\ell}b_{ij}b_{k\ell}((i\overline{j},k\overline{\ell})+(i\overline{\ell},k\overline{j})+(i\overline{j},\ell\overline{k})+(i\overline{k},\ell\overline{j}) \\
&\ \ \ \ \ +& (j\overline{i},k\overline{\ell})+(j\overline{\ell},k\overline{i})+(j\overline{i},\ell\overline{k})+(j\overline{k},\ell\overline{i}))\\
&=& -\sum_{i,j,k,\ell}b_{ij}b_{k\ell}(i\overline{j}+j\overline{i},k\overline{\ell}+\ell\overline{k})- \sum_{i,j,k,\ell}b_{ij}b_{k\ell}((i\overline{\ell},k\overline{j})+(\ell\overline{i},j\overline{k}))\\
&\ \ \ \ \ -& \sum_{i,j,k,\ell}b_{ij}b_{k\ell}((i\overline{k},\ell\overline{j})+(j\overline{\ell},k\overline{i})).
\eeqar
Let us work with these three terms. For the first one, we have 
\beqar
&& -\sum_{i,j,k,\ell}b_{ij}b_{k\ell}(i\overline{j}+j\overline{i},k\overline{\ell}+\ell\overline{k}) \\
&=& -\int_{X}D(\sum_{ij}b_{ij}\mu_{i}\overline{\mu_{j}}+\sum_{ij}b_{ij}\mu_{j}\overline{\mu_{i}})
(\sum_{ij}b_{ij}\mu_{i}\overline{\mu_{j}}+\sum_{ij}b_{ij}\mu_{j}\overline{\mu_{i}})dA(z)\\
&=& -\int_{X}{D(H(z,z)+\overline{H(z,z)})(H(z,z)+\overline{H(z,z)})}dA(z)\\
&=& -4\int_{\D}D(\Re H(z,z))(\Re H(z,z))dA(z). 
\eeqar
For the second term, using the same argument in calculating $\Q(A,A)$ above, we have
\beqar
&& - \sum_{i,j,k,\ell}b_{ij}b_{k\ell}((i\overline{\ell},k\overline{j})+(\ell\overline{i},j\overline{k})) \\
&=& -2\Re\{\iint_{\D\times\D}G(w,z)H(z,w)H(w,z)dA(w)dA(z)\}.
\eeqar
While similarly the third term yields
\beqar
&&-\sum_{i,j,k,\ell}b_{ij}b_{k\ell}((i\overline{k},\ell\overline{j})+(j\overline{\ell},k\overline{i})) \\
&=& -2\iint_{\D\times\D}G(w,z)|H(z,w)|^2dA(w)dA(z).
\eeqar
Thus,
\beqar
\Q(B,B)&=& -4\int_{\D}D(\Re H(z,z))(\Re H(z,z))dA(z) \\
&\ \ \ \ -&2\Re\{\iint_{\D\times\D}G(w,z)H(z,w)H(w,z)dA(w)dA(z)\} \\
&\ \ \ \ -&2\iint_{\D\times\D}G(w,z)|H(z,w)|^2dA(w)dA(z).
\eeqar
\ep

We are left to deal with the final expression $\Q(A,B)$.
\bl \label{3.4-3}
Using above notations, we have 
\beqar
\Q(A,B)&=& -4\int_{\D}{D(\Im F(z,z))(\Re H(z,z))}dA(z) \\
&\ \ \ \ \ -& 2\Im\{\iint_{\D\times\D}{G(z,w)F(z,w)H(w,z)}dA(w)A(z)\} \\
&\ \ \ \ \ -& 2\Im\{\iint_{\D\times\D}{G(z,w)F(z,w)\overline{H(z,w)}}dA(w)A(z)\}. 
\eeqar
\el
\bp [Proof of Lemma \ref{3.4-3}]
\beqar
\Q(A,B) &=&(-\bi)\sum_{i,j,k,\ell}a_{ij}b_{k\ell}(-R_{i\bar{j}k\bar{\ell}}+R_{i\bar{j}\bar{k}\ell}
-R_{\bar{i}jk\bar{\ell}}+R_{\bar{i}j\bar{k}\ell})\\
&=&(-\bi)\sum_{i,j,k,\ell}a_{ij}b_{k\ell}\{-(i\bar{j},k\bar{\ell})-(i\bar{\ell},k\bar{j})-(i\bar{j},\ell\bar{k})-(i\bar{k},\ell\bar{j})\\
&\ \ \ \ \  +&(j\bar{i},k\bar{\ell})+(j\bar{\ell},k\bar{i})+(j\bar{i},\ell\bar{k})+(j\overline{k},\ell\bar{i})\} \\
&=&(-\bi)\sum_{i,j,k,\ell}a_{ij}b_{k\ell}(j\bar{i}-i\bar{j},k\bar{\ell}+\ell\bar{k})\\
&\ \ \ \ \ +& (-\bi)\sum_{i,j,k,\ell}a_{ij}b_{k\ell}(-(i\bar{\ell},k\bar{j})+(\ell\bar{i},j\bar{k}))\\
&\ \ \ \ \ +& (-\bi)\sum_{i,j,k,\ell}a_{ij}b_{k\ell}(-(i\bar{k},\ell\bar{j})+(j\bar{\ell},k\bar{i}))\\
&=& -\bi(-2\bi)\int_{\D}{D(\Im F(z,z))(2\Re H(z,z))}dA(z) \\
&\ \ \ \ \ +& (-\bi)(-2\bi)\Im\{\iint_{\D\times\D}{G(z,w)F(z,w)H(w,z)}dA(w)A(z)\} \\
&\ \ \ \ \ +& (-\bi)(-2\bi)\Im\{\iint_{\D\times\D}{G(z,w)F(z,w)\overline{H(z,w)}}dA(w)A(z)\} \\
&=& -4\int_{\D}{D(\Im F(z,z))(\Re H(z,z))}dA(z) \\
&\ \ \ \ \ -& 2\Im\{\iint_{\D\times\D}{G(z,w)F(z,w)H(w,z)}dA(w)A(z)\} \\
&\ \ \ \ \ -& 2\Im\{\iint_{\D\times\D}{G(z,w)F(z,w)\overline{H(z,w)}}dA(w)A(z)\}. 
\eeqar
\ep

\bpro [Formula for curvature operator] \label{expression}
Using above notations, we have 
\begin{small}
\beqar
&&\Q(A+B,A+B) =\\
&-&4\int_{\D}{D (\Im\{F(z,z)+\bi H(z,z)\})\cdot(\Im\{F(z,z)+\bi H(z,z)\})dA(z)}\\
&-&2\iint_{\D\times \D}G(z,w)|F(z,w)+\bi H(z,w))|^2dA(w)dA(z)\\
&+&2\Re\{\iint_{\D\times \D}G(z,w)(F(z,w)+\bi H(z,w))(F(w,z)+\bi H(w,z))dA(w)dA(z)\},
\eeqar
\end{small}
where $F(z,w)$ and $H(z,w)$ are defined in \eqref{F} and \eqref{H} respectively, and $G(z,w)$ is 
the {\Gf} for the operator $D$.
\epro
\bp [Proof of Proposition ~\ref{expression}]
It follows from Lemma \ref{3.4-1}, \ref{3.4-2} and \ref{3.4-3} that
\begin{eqnarray*}
&& \Q(A+B,A+B)\\
&=&(\int_{\D}D(F(z,z)-\overline{F(z,z)})(F(z,z)-\overline{F(z,z)})dA(z)\\
             &-&\int_{\D}{D(H(z,z)+\overline{H(z,z)})(H(z,z)+\overline{H(z,z)})}dA(z)\\
             &+&2\textbf{i} \cdot \int_{\D}{D(F(z,z)-\overline{F(z,z)})(H(z,z)+\overline{H(z,z)})}dA(z))\\
            (&-&2\cdot \int_{\D\times \D}G(z,w)|F(z,w))|^2dA(w)dA(z)\\
             &-&2\cdot \int_{\D\times \D}G(z,w)|H(z,w))|^2dA(w)dA(z)\\            
             &-&4\cdot \Im\{\int_{\D\times \D}G(z,w)F(z,w)\overline{H(z,w)})dA(w)dA(z)\})\\
             (&+&2\cdot \Re\{\int_{\D\times \D}G(z,w)F(z,w)F(w,z)dA(w)dA(z)\}\\
             &-&2\cdot \Re\{\int_{\D\times \D}G(z,w)H(z,w)H(w,z)dA(w)dA(z)\}\\
             &-&4\cdot \Im\{\int_{\D\times \D}G(z,w)F(z,w)H(w,z)dA(w)dA(z)\}).
\end{eqnarray*}
The sum of the first three terms is exactly
$$
-4\int_{\D}{D(\Im\{F(z,z)+\textbf{i}H(z,z)\})\cdot(\Im\{F(z,z)+\textbf{i}H(z,z)\})dA(z)}.
$$

Just as $|a+\textbf{i}b|^2=|a|^2+|b|^2+2\cdot \Im(a\cdot \overline{b})$, where $a$ and $b$ are two complex numbers, the sum of the second three terms is
exactly
$$
-2\cdot \int_{\D\times \D}G(z,w)|F(z,w)+\textbf{i}H(z,w))|^2dA(w)dA(z).
$$

For the last three terms, since 
\begin{eqnarray*}
\Im(F(z,w)\cdot H(w,z))=-\Re(F(z,w)\cdot(\textbf{i}H(w,z))),
\end{eqnarray*} 
the sum is exactly 

\begin{small}
$$
2\cdot \Re\{\int_{\D\times \D}G(z,w)(F(z,w)+\textbf{i}H(z,w))(F(w,z)+\textbf{i}H(w,z))dA(w)dA(z)\}.
$$
\end{small}
The proof is complete.
\ep

Now we continue with the proof of Theorem ~\ref{npd}. By Proposition ~\ref{expression}, there are three integrals in the 
expression of $\Q(A+B,A+B)$. We first work with the last two terms 
by the {\csi} and the fact that $G(z,w)=G(w,z)$ to find:
\begin{small}
\beqar
&&|\iint_{\D\times \D}G(z,w)(F(z,w)+\bi H(z,w))(F(w,z)+\bi H(w,z))dA(w)dA(z)| \\
&\le& \iint_{\D\times \D}G(z,w)|(F(z,w)+\bi H(z,w))(F(w,z)+\bi H(w,z))|dA(w)dA(z)\\
&\le& \{\iint_{\D\times \D}G(z,w)|F(z,w)+\bi H(z,w)|^2 dA(w)dA(z)\}^{\frac12}\\
&\ \ \ \ \ \times& \{\iint_{\D\times \D}G(w,z)|F(w,z)+\bi H(w,z)|^2 dA(w)dA(z)\}^{\frac12}\\
&=& \iint_{\D\times \D}G(w,z)|F(z,w)+\bi H(z,w)|^2 dA(z)dA(w).
\eeqar
\end{small}
Thus, we have
\bpro \label{m-pro}
$$\Q(A+B,A+B)\leq -4\int_{\D}D(\Im F(z,z))(\Im F(z,z))dA(z).$$
\epro

Theorem ~\ref{npd} follows directly from Proposition ~\ref{D1} and ~\ref{m-pro}. 
\ep

\subsection{Zero level set}
In order to apply Theorem ~\ref{npd} to more geometrical situations later, we determine the zero level set for the operator $\Q$.

First let us define an action on $\wedge^2TT_H(1)$. Recall our explicit {\ob} from \eqref{ob}:
\beq
\mu_{n-1} = \frac{(1-|z|^2)^2}{4}\sqrt{\frac{2n^3-2n}{\pi}}{\bar{z}}^{n-2}, \ \ \ n\ge 2.
\eeq
For any point $P \in T_H(1)$, let $\{\ppl{}{t_j}\}_{j\ge 1}$ be the vector field on $T_H(1)$ near $P$ such that $\ppl{}{t_j}|_P = \mu_j$, 
and we write $t_j = x_j+\bi y_j$, then the complex structure $\J$ associated with the {\wpm} is an isometry on the {\ts} $\hd$ with 
${\J} \ppl{}{x_i} = \ppl{}{y_i}$ and ${\J}^2 = -\text{id}$. This naturally extends to an action, which we abuse our notation to denote 
it by $\J$, on $\wedge^2TT_H(1)$. 
\begin{Def}\label{J}
The action $(\J, \circ)$ is defined as follows on a basis:
\beqar
\begin{cases}
\textbf{J}\circ\ppl{}{x_i}\wedge \ppl{}{x_j}:=\ppl{}{y_i}\wedge \ppl{}{y_j},\\
\textbf{J}\circ\ppl{}{x_i}\wedge \ppl{}{y_j}:=-\ppl{}{y_i}\wedge \ppl{}{x_j}=\ppl{}{x_j}\wedge \ppl{}{y_i},\\
\textbf{J}\circ\ppl{}{y_i}\wedge \ppl{}{y_j}:=\ppl{}{x_i}\wedge \ppl{}{x_j},
\end{cases}
\eeqar
and we extend it linearly. 
\end{Def}
\bl
We have $\J\circ\J=\text{id}$. Moreover, 
\be\label{act}
\langle\ppl{}{x_i}, \ppl{}{x_j}\rangle(P) = \langle\ppl{}{y_i}, \ppl{}{y_j}\rangle(P) = \delta_{ij}, \quad\ 
\langle\ppl{}{x_i}, \ppl{}{y_j}\rangle(P) = 0.
\ene
\el
\bp
The identity $\J\circ\J=\text{id}$ is clear by definition. We only show the first equality in \eqref{act}:
\beqar
\langle\ppl{}{x_i}, \ppl{}{x_j}\rangle(P) &=& \Re\{\langle\ppl{}{x_i}+
\bi\J\ppl{}{x_i}, \ppl{}{x_j}+\bi\J\ppl{}{x_j}\rangle(P)\}\\
&=& \Re\{\langle\ppl{}{\mu_i}, \ppl{}{\mu_j}\rangle(P) \}=\delta_{ij}.
\eeqar
\ep
Now we treat the equality case for $\Q \le 0$, namely, 
\bt\label{zero}(= (ii) of Theorem \ref{mt-1})
For $C \in \wedge^2TT_H(1)$, $\Q(C, C) = 0$ if and only if there is an element $E \in \wedge^2TT_H(1)$ such that $C = E - \J\circ E$, 
where $\J\circ$ is defined above. 
\et
\bp
One direction is straightforward: If $C = E - \J\circ E$ for some $E \in \wedge^2TT_H(1)$, we have that $\Q(C, C) = 0$ since $\J$ is an 
isometry on tangent space.

Conversely, let $C \in \wedge^2TT_H(1)$ with $\Q(C, C) = 0$. We write 
$$C=\sum_{ij}a_{ij}\frac{\partial}{\partial x_{i}}\wedge \frac{\partial}{\partial x_{j}}
+b_{ij}\frac{\partial}{\partial x_{i}}\wedge \frac{\partial}{\partial y_{j}}+ c_{ij}\frac{\partial}{\partial y_{i}}\wedge \frac{\partial}{\partial y_{j}}.$$
Applying the identities in the proof of Lemma ~\ref{reduction}, we have 
\be\label{QCC}
\Q(C,C) =\Q(\sum_{ij}d_{ij}\frac{\partial}{\partial x_{i}}\wedge \frac{\partial}{\partial x_{j}}
+b_{ij}\frac{\partial}{\partial x_{i}}\wedge \frac{\partial}{\partial y_{j}},\sum_{ij}d_{ij}\frac{\partial}{\partial x_{i}}\wedge \frac{\partial}{\partial x_{j}}
+b_{ij}\frac{\partial}{\partial x_{i}}\wedge \frac{\partial}{\partial y_{j}}), 
\ene
where $d_{ij}=a_{ij}+c_{ij}$. This enables us to write $C = A+B$, where 
$$A = \sum\limits_{i j}{d_{ij}\ppl{}{x_i}\wedge \ppl{}{x_j}}, \ \ B=\sum\limits_{i j}{b_{ij}\ppl{}{x_i}\wedge \ppl{}{y_j}}.$$
From the proof of $\Q(A+B,A+B) \le 0$ in Proposition ~\ref{expression}, we find that $Q(A+B,A+B)=0$ if and only if there exists a constant $k$ 
such that both of the following hold:
\begin{eqnarray*}
\begin{cases}
\Im\{F(z,z)+\textbf{i}H(z,z)\}=0, \\
F(z,w)+\textbf{i}H(z,w)=k\cdot\overline{(F(w,z)+\textbf{i}H(w,z))}.
\end{cases}
\end{eqnarray*}
Setting $z=w$, we find $k=1$. Therefore the second equation above implies 
\begin{eqnarray*}
\sum_{ij}(d_{ij}-d_{ji}+\textbf{i}(b_{ij}+b_{ji}))\mu_{i}(w)\overline{\mu_{j}}(z)=0.
\end{eqnarray*} 
Since $\{\mu_{i}\}_{i\geq 1}$ is a basis,
\begin{eqnarray*}
d_{ij}=d_{ji},\quad \quad b_{ij}=-b_{ji}.
\end{eqnarray*}
That is,
\beq
a_{ij}-a_{ji}=-(c_{ij}-c_{ji}), \ \ \ b_{ij}=-b_{ji}.
\eeq
We now define 
\be\label{E}
E=\sum_{ij}a_{ij}\frac{\partial}{\partial x_{i}}\wedge \frac{\partial}{\partial x_{j}}+\frac{b_{ij}}{2}\frac{\partial}{\partial x_{i}}\wedge \frac{\partial}{\partial y_{j}},
\ene
and we verify that $C=E-\textbf{J}\circ E$. Indeed, first we have 
\beq
\sum_{ij} a_{ij}\frac{\partial}{\partial x_{i}}\wedge \frac{\partial}{\partial x_{j}}=\sum_{i<j}(a_{ij}-a_{ji})\frac{\partial}{\partial x_{i}}\wedge \frac{\partial}{\partial x_{j}},
\eeq
then we apply the definition of $\J\circ$ in Definition ~\ref{J} to find
\beqar
\J\circ \sum_{ij}a_{ij}\frac{\partial}{\partial x_{i}}\wedge \frac{\partial}{\partial x_{j}}&=&\sum_{i<j}(a_{ij}-a_{ji})\frac{\partial}{\partial y_{i}}\wedge \frac{\partial}{\partial y_{j}}\\
&=&-\sum_{i<j}(c_{ij}-c_{ji})\frac{\partial}{\partial y_{i}}\wedge \frac{\partial}{\partial y_{j}}\\
&=&-\sum{c_{ij}\frac{\partial}{\partial y_{i}}\wedge \frac{\partial}{\partial y_{j}}}.
\eeqar
Similarly, 
\beq
\J\circ\sum(\frac{b_{ij}}{2}\frac{\partial}{\partial x_{i}}\wedge \frac{\partial}{\partial y_{j}})=-\sum\frac{b_{ij}}{2}\frac{\partial}{\partial x_{i}}\wedge \frac{\partial}{\partial y_{j}}.
\eeq
This completes the proof.

\ep


\section{Boundedness}
If we denote $\langle\cdot,\cdot\rangle$ the pairing of vectors in the space $\wedge^2TT_H(1)$. This 
natural inner product on $\wedge^2TT_H(1)$ associated to the {\wpm} on $T_H(1)$ is given as (following 
\cite{Lan99}): $\forall V_i \in \hd$,
\be\label{eu}
\langle V_1\wedge V_2, V_3\wedge V_4\rangle_{eu} = 
\langle V_1, V_3\rangle \langle V_2, V_4\rangle - \langle V_1, V_4\rangle \langle V_2, V_3\rangle.
\ene

The goal of this section is to prove Theorem ~\ref{mt-2}, namely,
\bt\label{bd}(=Theorem \ref{mt-2})
The {\co} $\Q$ is bounded, i.e., for any $V \in \wedge^2TT_H(1)$ with $\|V\|_{eu} =1$, we have 
$|\Q(V,V)|\le 16\sqrt{\frac{3}{\pi}}$, where $\|\cdot\|_{eu}$ is the Euclidean norm for the wedge product defined in \eqref{eu}. 
\et
We follow the idea for the proof of Theorem 1.3 in \cite{WW15}.
\subsection{Technical lemmas}
Let us begin with a useful lemma.
\bl
Let $D=-2(\Delta_\rho-2)^{-1}$ as above above. Then, for any complex-valued function $f\in L^{2}(\D)\cap BC^{\infty}(\D)$,
\be\label{Dff}
\int_{\D}{(D(f)\bar{f})dA}\leq \int_{\D}{|f|^2 dA},
\ene
where $dA = \rho|dz|^2$ is the hyperbolic area element for $\D$.
\el
\bp
Recall the {\Gf} of $D$ on $\D$ is $G(z,w)$, such that, $\forall f \in L^2(\D, \C)$, we have 
\be\label{greenD}
D(f)(z) = \int_{w\in\D}G(z,w)f(w)dA(w).
\ene
Assuming first $f$ is real valued, we apply the {\csi} and the symmetry of the {\Gf}:
\begin{small}
\beqar
&&\int_{\D}{D(f(z))f(z)dA(z)}\\
&=& \iint_{\D\times\D}{\{G(z,w)f(w)dA(w)\}f(z)dA(z)}\\
&\le& \sqrt{\iint_{\D\times\D}{G(z,w)f^2(w)dA(w)dA(z)}}\cdot\sqrt{\iint_{\D\times\D}{G(z,w)f^2(z)dA(z)dA(w)}} \\
&=& \sqrt{\iint_{\D}{D(f^2(w))dA(w)}}\cdot\sqrt{\int_{\D}{D(f^2(z))dA(z)}}\\
&=& \int_{\D}{D(f^2(z))dA(z)} \\
&=& \int_{\D}f^2(z)dA(z).
\eeqar
\end{small}
When $f$ is complex valued, we can write $f = f_1 + \bi f_2$, where $f_1$ and $f_2$ are real-valued. Then 
using $D$ is self-adjoint (Proposition ~\ref{D1}), we find,
\beqar
\int_{\D}{(D(f)\bar{f})dA} &=& \int_{\D}{(D(f_1)f_1dA} + \int_{\D}{(D(f_2)f_2dA} \\
&\le& \int_{\D}|f_1|^2 + |f_2|^2 dA \\
&=& \int_{\D}{|f|^2dA}.
\eeqar
\ep
Recalling from \eqref{F} and \eqref{H}, we write $F(z,z) = \sum\limits_{ij} a_{ij}\mu_i(z)\bar{\mu}_j(z)$ for the expression 
$A = \sum\limits_{i,j\ge 1}{a_{ij}\ppl{}{x_i}\wedge \ppl{}{x_j}}$, and 
$H(z,z) = \sum\limits_{ij} b_{jj}\mu_i(z)\bar{\mu}_j(z)$ for the expression $B = \sum\limits_{i,j\ge 1}{b_{ij}\ppl{}{x_i}\wedge \ppl{}{y_j}}$, 
where $\{\mu_i\}_{i \ge 1}$ is the {\ob} \eqref{ob} for $\hd$. By Lemma ~\ref{reduction}, to show Theorem ~\ref{bd}, it suffices to work with $V= A+B$.
\bl\label{QVV}
Under above notation, we have the following estimate:
\begin{small}
\begin{eqnarray}\label{bd-qvv}
|\Q(V,V)| &\le& 8\cdot(\int_{\D}|F(z,z)|^2dA+ \int_{\D}|H(z,z)|^2dA) \nonumber\\
&\ \ \ \ \ \ \ +& 4\iint_{\D\times \D}G(z,w)|F(z,w)+\bi H(z,w))|^2dA(w)dA(z).
\end{eqnarray}
\end{small}
\el
\bp
The expression for $\Q(V,V)$ is shown in Proposition ~\ref{expression}. We have
\begin{small}
\beqar
&&|\Q(V,V)| \\
&\le& 4\int_{\D}{D (\Im\{F(z,z)+\bi H(z,z)\})\cdot(\Im\{F(z,z)+\bi H(z,z)\})dA(z)}\\
&\ \ \ \ +&2\iint_{\D\times \D}G(z,w)|F(z,w)+\bi H(z,w))|^2dA(w)dA(z)\\
&\ \ +&2\Re\{\iint_{\D\times \D}G(z,w)(F(z,w)+\bi H(z,w))(F(w,z)+\bi H(w,z))dA(w)dA(z)\}.
\eeqar
\end{small}
We work with these terms. First we apply Lemma ~\ref{Dff}, and the triangle inequality to find:
\beqar
&&4\int_{\D}{D (\Im\{F(z,z)+\bi H(z,z)\})\cdot(\Im\{F(z,z)+\bi H(z,z)\})dA(z)} \\
&\ \ \ \ \ \le& 4\int_{\D}{|\Im\{F(z,z)+\bi H(z,z)\}|^2dA(z)}\\
&\ \ \ \ \ \le& 8(\int_{\D}{|\{F(z,z)|^2dA+\int_\D|H(z,z)\}|^2dA}).
\eeqar
Recalling from the end of the proof of Theorem ~\ref{npd}, we have established the following inequality:
\begin{small}
\beqar
&&|\iint_{\D\times \D}G(z,w)(F(z,w)+\bi H(z,w))(F(w,z)+\bi H(w,z))dA(w)dA(z)| \\
&\le& \iint_{\D\times \D}G(w,z)|F(z,w)+\bi H(z,w)|^2 dA(z)dA(w).
\eeqar
\end{small}
Now \eqref{bd-qvv} follows.
\ep
We will also quote a Harnack type inequality for any $\mu(z) \in \hd$. We note that it works in our favor 
that the injectivity radius of $\D$ is infinity.

\bpro\cite{TT06, Teo09}\label{harnack}
Let $\mu \in \hd$. Then the  $L^\infty$-norm of $\mu \in \hd$ can be estimated from above by its {\WP} norm, namely, for all 
$\mu \in \hd$, we have 
\be\label{harnack2}
sup_{z\in\D}|\mu(z)| \le \sqrt{\frac{3}{4\pi}}\|\mu\|_{WP}.
\ene
\epro
We also derive the following estimate which is quite general, and we formulate it to the 
following lemma:
\bl\label{K}
Let $\{z\}, \{w\}$ two complex coordinates on $\D$, if a converging series is in the form of 
\beq
K(z,w) = \sum\limits_{i,j\ge 1} d_{ij}\mu_i(w)\overline{\mu_j(z)},
\eeq
for some $d_{ij} \in \R$, where $\{\mu_j\}$ is the {\ob} \eqref{ob} on $\hd$, then we have 
\be\label{K1}
\int_\D|K(z,z)|^2dA(z) \le \sqrt{\frac{3}{4\pi}}\sum\limits_{i,j\ge 1}d_{ij}^2, 
\ene
and
\be\label{K2}
sup_{w\in\D}|K(z,w)|^2  \le \sqrt{\frac{3}{4\pi}}\sum\limits_{i,j, \ell\ge 1}d_{ij}d_{i\ell}\overline{\mu_j(z)}\mu_\ell(z).
\ene
\el
\bp We use the standard technique for this type of argument, namely, since 
\beq
|K(z,z)|^2 \le sup_{w\in\D}|K(z,w)|^2,
\eeq
we will try to use one complex coordinate against the other. Fixing $z$, we note that the form $K(z,w)$ is a 
harmonic {\Bd} on $\D$ in the coordinate $w$. Indeed, 
\beq
K(z,w) = \sum\limits_{i\ge 1}\{\sum\limits_{j\ge 1}d_{ij}\overline{\mu_j(z)}\}\mu_i(w).
\eeq
This enables us to apply \eqref{harnack2}:
\begin{small}
\beqar
sup_{w\in\D}|K(z,w)|^2 &\le& \sqrt{\frac{3}{4\pi}}\int_\D K(z,w)\overline{K(z,w)}dA(w)\\
&=& \sqrt{\frac{3}{4\pi}}\int_\D \{\sum\limits_{i,j}d_{ij}\mu_i(w)\overline{\mu_j(z)}\}
\{\sum\limits_{k,\ell}d_{k\ell}\overline{\mu_k(w)}\mu_\ell(z)\}dA(w)\\
&=& \sqrt{\frac{3}{4\pi}}\{\sum\limits_{i,j,\ell}d_{ij}d_{i\ell}\overline{\mu_j(z)}\mu_\ell(z)\}.
\eeqar
\end{small}
We also used the basis $\{\mu_j\}$ is orthonormal {\wrt} the {\wpm}. Therefore we have  
\beqar
\int_\D|K(z,z)|^2dA(z) &\le& \int_\D sup_{w\in\D}|K(z,w)|^2dA(z)\\
 &\le& \sqrt{\frac{3}{4\pi}}\int_\D \sum\limits_{i,j,\ell}d_{ij}d_{i\ell}\overline{\mu_j(z)}\mu_\ell(z)dA(z)\\
 &=& \sqrt{\frac{3}{4\pi}}\sum\limits_{i,j}d_{ij}^2.
\eeqar
\ep


\subsection{$\Q$ is bounded}
We now prove the boundedness. 
\bp[Proof of Theorem ~\ref{bd}]
We find via the definition of the Euclidean inner product \eqref{eu} and the symmetric properties of the {\ct}: 
\be\label{delta1}
\langle\ppl{}{x_i}\wedge \ppl{}{x_j}, \ppl{}{x_k}\wedge \ppl{}{x_\ell}\rangle_{eu}(P)=
\delta_{ik}\delta_{j\ell}-\delta_{i\ell}\delta_{jk},
\ene
and
\be\label{delta1-1}
\langle\ppl{}{y_i}\wedge \ppl{}{y_j}, \ppl{}{y_k}\wedge \ppl{}{y_\ell}\rangle_{eu}(P)=
\delta_{ik}\delta_{j\ell}-\delta_{i\ell}\delta_{jk},
\ene
and 
\be\label{delta2}
\langle\ppl{}{x_i}\wedge \ppl{}{y_i}, \ppl{}{x_j}\wedge \ppl{}{y_j}\rangle_{eu}(P)=\delta_{ij},
\ene
and 
\be\label{delta3}
\langle\ppl{}{x_i}\wedge \ppl{}{x_j}, \ppl{}{x_k}\wedge \ppl{}{y_\ell}\rangle_{eu}(P)=0.
\ene

We now denote $V = A + B+ C \in \wedge^2TT_H(1)$, where 
$A = \sum\limits_{i<j}{a_{ij}\ppl{}{x_i}\wedge \ppl{}{x_j}}$, 
$B = \sum\limits_{i,j\ge 1}{b_{ij}\ppl{}{x_i}\wedge \ppl{}{y_j}}$, and
$C = \sum\limits_{i<j}{c_{ij}\ppl{}{y_i}\wedge \ppl{}{y_j}}$. Then by \eqref{delta1}, \eqref{delta1-1}, \eqref{delta2}, and 
\eqref{delta3}, we have 

\be
\langle A, A\rangle_{eu}(P)= \sum\limits_{i<j}a_{ij}^2,  \ \ \ 
 \langle B, B\rangle_{eu}(P)= \sum\limits_{i,j}b_{ij}^2, \ \ \  \langle C, C\rangle_{eu}(P)=\sum\limits_{i<j}c_{ij}^2.
\ene 
and
\be
\langle A, B\rangle_{eu}(P)=0, \ \ \  \langle A, C\rangle_{eu}(P)=0, \ \ \ 
 \langle B, C\rangle_{eu}(P)=0.
\ene 

Assume that $\|V\|_{eu} =1$, that is
\[\sum\limits_{i<j}a_{ij}^2+\sum\limits_{i,j}b_{ij}^2+\sum\limits_{i<j}c_{ij}^2=1. \]

Recalling from Lemma ~\ref{reduction} and ~\ref{QVV}, we have:
\begin{small}
\beqar
|\Q(V,V)| &\le& 8\cdot(\int_{\D}|F(z,z)|^2dA+ \int_{\D}|H(z,z)|^2dA) \\
&\ \ \ \ \ \ \ +& 4\iint_{\D\times \D}G(z,w)|F(z,w)+\bi H(z,w))|^2dA(w)dA(z).
\eeqar
Where $F(z,w) = \sum\limits_{i<j}(a_{ij}+c_{ij})\overline{\mu_j(z)}\}\mu_i(w)$ and $
H(z,w) = \sum\limits_{i,j\ge 1}b_{ij}\overline{\mu_j(z)}\}\mu_i(w).$
\end{small}
We now estimate these two terms. Note that both our forms $F(z,w)$ (\eqref{F}) and $H(z,w)$ (\eqref{H}) are of the 
type in Lemma ~\ref{K}. Since both series $\sum\limits_{i<j}a_{ij}^2$ and $\sum\limits_{i,j}b_{ij}^2$ 
converge, we have the first term in \eqref{bd-qvv} bounded from above as follows:
\begin{small}
\beqar
8\cdot(\int_{\D}|F(z,z)|^2dA+ \int_{\D}|H(z,z)|^2dA) &\le& 8\sqrt{\frac{3}{4\pi}}( \sum\limits_{i<j}(a_{ij}+c_{ij})^2+  \sum\limits_{i,j}b_{ij}^2)\\
&\leq& 8\sqrt{\frac{3}{4\pi}} \times 2 (\sum\limits_{i<j}a_{ij}^2+\sum\limits_{i,j}b_{ij}^2+\sum\limits_{i<j}c_{ij}^2) \\
&=& 8\sqrt{\frac{3}{\pi}}.
\eeqar
\end{small}
 The second term is also bounded by applying \eqref{K2}. To see this, we use the fact that $\{\mu_i\}$'s form an {\ob}, 
 \beqar
&&\iint_{\D\times \D}G(z,w)|F(z,w)|^2dA(z)dA(w) \\
&\le& \iint_{\D\times \D}G(z,w)sup_{w\in\D}(|F(z,w)|^2)dA(z)dA(w)\\
&\le& \iint_{\D\times \D}G(z,w)\sqrt{\frac{3}{4\pi}}\sum_{i<j}\sum\limits_{\ell\ge 1}(a_{ij}+c_{ij})(a_{i\ell}+c_{i\ell})\overline{\mu_j(z)}\mu_\ell(z)dA(z)dA(w)\\
&=& \int_{\D}D(\sqrt{\frac{3}{4\pi}}\sum_{i<j}\sum\limits_{\ell\ge 1}(a_{ij}+c_{ij})(a_{i\ell}+c_{i\ell})\overline{\mu_j(z)}\mu_\ell(z))dA(w)\\
&=& \int_{\D}\sqrt{\frac{3}{4\pi}}\sum_{i<j}\sum\limits_{\ell\ge 1}(a_{ij}+c_{ij})(a_{i\ell}+c_{i\ell})\overline{\mu_j(z)}\mu_\ell(z)dA(w) \\
&=& \sqrt{\frac{3}{4\pi}}\sum\limits_{i<j}(a_{ij}+c_{ij})^2.
\eeqar
Similar argument yields $$\iint_{\D\times \D}G(z,w)|H(z,w)|^2dA(z)dA(w) \le \sqrt{\frac{3}{4\pi}}\sum\limits_{i,j}b_{ij}^2.$$ 
Therefore the second term can be estimated as follows:
\beqar
&&4\iint_{\D\times \D}G(z,w)|F(z,w)+\bi H(z,w))|^2dA(w)dA(z) \\
&\le& 8\iint_{\D\times \D}G(z,w)|F(z,w)|^2dA(z)dA(w)\\
&\ \ \ \ \ \ \ \ \ \ \ \ \ \ \  +& 8\iint_{\D\times \D}G(z,w)|H(z,w)|^2dA(z)dA(w)\\
&\le&8\sqrt{\frac{3}{4\pi}}( \sum\limits_{i<j}(a_{ij}+c_{ij})^2+  \sum\limits_{i,j}b_{ij}^2) \\
&\leq& 8\sqrt{\frac{3}{\pi}}.
\eeqar
Combining with earlier same upper bound for the first term, we find $$|\Q(V,V)| \le 16\sqrt{\frac{3}{\pi}}.$$ 
Proof is now complete.
\ep

\section{Noncompactness}
In this section, we treat the question about the compactness.
\bt\label{noncpt}(=Theorem \ref{mt-3})
The {\co} $\Q$ is not a compact operator. 
\et

We will prove this theorem by contradiction. First we proceed with several technical lemmas.

\bl\label{4-1}
For any harmonic Beltrami differential $\mu \in L^2(\D)$ on $\D$, we have
\be
D(|\mu|^2)\geq \frac{|\mu|^2}{3}.
\ene
\el
\bp
The argument here is motivated by Lemma 5.1 in \cite{Wlf12} which is for the case of a closed Riemann surface.


Recall that the curvature of a metric expressed as $\sigma(z)|dz|^2$ on a Riemannian 2-manifold is given by
$$K(\sigma(z)|dz|^2)=-\frac{1}{2}\Delta_{\sigma}\ln(\sigma(z)),$$
where $\Delta_{\sigma}$ is the Laplace-Beltrami operator of $\sigma(z)|dz|^2$.

Suppose that $p \in \D$ with $|\mu|\neq 0$. By definition one may assume that $$|\mu(z)|=\frac{|\Phi(z)|}{\rho(z)}$$ where $\Phi(z)$ is 
holomorphic on $\D$. Let $\Delta$ be the Laplace-Beltrami operator of $\rho(z)|dz|^2$. Then using the curvature information that 
$K(\rho(z)|dz|^2)=-1$ and $K(|\Phi(z)||dz|^2)(p)=0$, we see that at $p\in \D$,
\beq
\Delta \ln \frac{|\Phi(p)|^2}{\rho^2(p)}=-4.
\eeq

On the other hand, at $p\in \D$ we have
\beq
\Delta \ln \frac{|\Phi(p)|^2}{\rho^2(p)}=\frac{\Delta \frac{|\Phi(p)|^2}{\rho^2(p)}}{\frac{|\Phi(p)|^2}{\rho^2(p)}}- 
\frac{|\nabla \frac{|\Phi(p)|^2}{\rho^2(p)}|^2}{\frac{|\Phi(p)|^4}{\rho^4(p)}}.
\eeq

Thus, at $p\in \D$ with $\mu(p) \not= 0$, we have
\be
\Delta \frac{|\Phi(p)|^2}{\rho^2(p)} \geq -4 \frac{|\Phi(p)|^2}{\rho^2(p)}.
\ene

If $p \in \D$ with $|\mu(p)|=0$, the maximum principle gives that
\beq
\Delta \frac{|\Phi(p)|^2}{\rho^2(p)} \geq 0= -4 \frac{|\Phi(p)|^2}{\rho^2(p)}.
\eeq

Therefore, we have
\be
 \Delta \frac{|\Phi(z)|^2}{\rho^2(z)} \geq -4 \frac{|\Phi(z)|^2}{\rho^2(z)}, \quad \ \  \forall z \in \D.
 \ene

Rewrite it as
\beq 
(\Delta-2) \frac{|\Phi(z)|^2}{\rho^2(z)} \geq -6 \frac{|\Phi(z)|^2}{\rho^2(z)}, \quad \forall z \in \D.
\eeq

Since the operator $D=-\frac{1}{2}(\Delta-2)^{-1}$ is positive on $\D$, the conclusion follows.
\ep
We also need the following elementary estimate:
\bl\label{4-3} 
For all positive integer $m\in \mathbb{Z}^+$, we have 
\be
\int_{0}^{1}{(1-r)^6\cdot r^m}dr\geq \frac{45}{2^{17}m^7}.
\ene
\el

\bp
The integral is the well-known beta function. Since $m$ is a positive integer,
\begin{eqnarray*}
\int_{0}^{1}{(1-r)^6\cdot r^m}dr &=& \frac{6! \cdot m!}{(m+7)!} \\
&\geq & \frac{6!}{(8m)^7} \\
&= & \frac{45}{2^{17}m^7}.
\end{eqnarray*}
\ep

We denote
\be\label{Ai}
A_i=\frac{1}{2^{\frac{i}{2}}}(\sum_{k=2^i}^{2^{i+1}-1}\frac{\partial}{\partial x_{k}}\wedge \frac{\partial}{\partial y_{k}}).
\ene 

\bpro\label{4-4}
For all $i$ large enough, we have 
\beq
-\Q(A_i,A_i)\geq 2^{-30}.
\eeq
\epro
\bp
We first observe that, from the definition \eqref{eu} of the inner product on $\wedge^2TT_H(1)$,  
$$
\langle A_i,A_j\rangle=\delta_{ij}.
$$
We wish to estimate $\Q(A_i,A_i)$ according to Proposition ~\ref{expression}. For $A_{i}$ in above \eqref{Ai}, by the definitions in \eqref{F} 
and \eqref{H}, the corresponding expressions $F$ and $H$ in Proposition \ref{m-pro} satisfy that
$$F(z,z)=0$$ and $$H(z,z)=\frac{1}{2^{\frac{i}{2}}}(\sum_{k=2^i}^{2^{i+1}-1} |\mu_k(z)|^2)$$
where $\mu_k(z)= \frac{(1-|z|^2)^2}{4}\sqrt{\frac{2(k+1)^3-2(k+1)}{\pi}}{\bar{z}}^{k-1}$.

It follows from Proposition ~\ref{expression} that
\begin{eqnarray*}
-\Q(A_i,A_i)&=&\frac{4}{2^i} \int_{\D}D(\sum_{k=2^i}^{2^{i+1}-1} |\mu_k(z)|^2) \cdot (\sum_{k=2^i}^{2^{i+1}-1} |\mu_k(z)|^2) \rho(z)|dz|^2\\
&\geq& \frac{4}{3\cdot 2^i} \int_{\D}(\sum_{k=2^i}^{2^{i+1}-1} |\mu_k(z)|^2)^2 \rho(z)|dz|^2,
\end{eqnarray*}
where we apply Lemma \ref{4-1} for the last inequality.

Now we apply the explicit expression of $$\mu_k(z)= \frac{(1-|z|^2)^2}{4}\sqrt{\frac{2(k+1)^3-2(k+1)}{\pi}}{\bar{z}}^{k-1},$$ and the fact that 
$2(k+1)^3-2(k+1)\geq (k+1)^3$,
\begin{small}
\begin{eqnarray*}
-\Q(A_i,A_i)&\geq& \frac{4}{3\cdot 2^i} \sum_{2^i\leq k, j \leq 2^{i+1}-1} \int_{\D}\frac{(k+1)^3(j+1)^3}{4^3\cdot \pi^2}(1-|z|^2)^6 |z|^{2k+2j-4}|dz|^2\\
&=& \frac{2\pi}{48\pi^2\cdot 2^i} \sum_{2^i\leq k, j \leq 2^{i+1}-1}(k+1)^3(j+1)^3 \int_0^{1}(1-r^2)^6r^{2k+2j-4}rdr \\
&=&  \frac{1}{48\pi\cdot 2^i} \sum_{2^i\leq k, j \leq 2^{i+1}-1}(k+1)^3(j+1)^3 \int_0^{1}(1-r)^6r^{k+j-2} dr \\
&\geq & \frac{1}{48\pi\cdot 2^i} \cdot \frac{45}{2^{17}}\sum_{2^i\leq k, j \leq 2^{i+1}-1} \frac{(k+1)^3(j+1)^3}{(k+j-2)^{7}} \quad \textit{(by Lemma \ref{4-3})} \\
&\geq & \frac{1}{48\pi\cdot 2^i} \cdot \frac{45}{2^{17}}\sum_{2^i\leq k, j \leq 2^{i+1}-1} \frac{(k+1)^3(j+1)^3}{4^7\cdot 2^{7i}},
\end{eqnarray*}
\end{small}
where the last inequality follows by $k+j-2\leq  2\cdot(2^{i+1}-1)-2 \leq 4\cdot 2^i$.

We simplify it as
\begin{eqnarray*}
-\Q(A_i,A_i)&\geq & \frac{15}{\pi\cdot 2^{35}} \cdot  \frac{(\sum_{2^i\leq k \leq 2^{i+1}-1}(k+1)^3)^2}{ 2^{8i}}
\end{eqnarray*}

It follows by an elementary formula
$$1^3+2^3+3^3+\cdots +m^3=(m(m+1))^2/4$$
that we have 
\beq
\sum_{2^i\leq k \leq 2^{i+1}-1}(k+1)^3=\frac{15}{4}\cdot 2^{4i}+O(2^{3i}).
\eeq
In particular, for large enough $i$ we may assume that
$$\sum_{2^i\leq k \leq 2^{i+1}-1}(k+1)^3>3\cdot 2^{4i}.$$

Therefore, for large enough $i$ we have
\beq
-\Q(A_i,A_i)\geq \frac{135}{\pi\cdot 2^{35}}>2^{-30}.
\eeq
This completes the proof.
\ep
Let us define $L:=Span\{A_n,A_{n+1},\cdots, A_{2n-1}\}$ which is an $n$-dimensional linear subspace in $\wedge^{2}TT_H(1)$ 
and $P_{L}:\wedge^{2}TT_H(1) \to L$ be the projection map. It is clear that $\langle A_k,A_\ell\rangle=\delta_{k\ell}$.
\bl\label{4-5} For the map
\[P_{L}\circ \Q: L \to L,\]
we have the following:
\ben
\item
$P_{L}\circ \Q$ is self-adjoint. 
\item
$P_{L}\circ \Q$ is non-positive definite. 
\item
$\sup_{A \in L, ||A||_{eu}=1} -\langle P_{L}\circ \Q(A),A\rangle\leq 16 \sqrt{\frac{3}{\pi}}$.
\een
\el
\bp
(i). Let $A, B \in L$. Then $P_L(A)=A$ and $P_L(B)=B$. Since $P_L$ is self-adjoint, we find 
\begin{eqnarray*}
\langle P_{L}\circ \Q(A),B\rangle&=&\langle\Q(A), P_L (B)\rangle\\
&=&\langle\Q(A),B\rangle=\langle A,\Q(B)\rangle\\
&=&\langle P_L(A),\Q(B)\rangle\\
&=&\langle A,P_L \circ \Q(B)\rangle.
\end{eqnarray*}

Thus, $P_L\circ \Q$ is self-adjoint.\\

(ii). For all $A\in L$,
\begin{eqnarray*}
\langle P_{L}\circ \Q(A),A\rangle&=&\langle\Q(A), P_L (A)\rangle\\
&=&\langle\Q(A),A\rangle \leq 0,
\end{eqnarray*}
where we apply the non-positivity of $\Q$ in the last inequality. Thus, $P_L\circ \Q$ is non-positive definite.\\

(iii). For all $A\in L$ with $||A||_{eu}=1$,
\begin{eqnarray*}
-\langle P_{L}\circ \Q(A),A\rangle&=&-\langle \Q(A), P_{L}(A)\rangle\\
&=&-\langle \Q(A),A\rangle\leq  16 \sqrt{\frac{3}{\pi}},
\end{eqnarray*}
where we apply Theorem \ref{mt-2} for the last inequality.
\ep
We now prove the main theorem in this section:

\begin{proof}[Proof of Theorem \ref{noncpt}]
Assume that $\Q$ is a compact operator. Then the well-known Spectral Theorem for compact operators \cite[Theorem A.3]{KSchbook} 
guarantees that there exists a sequence of eigenvalues $\{\sigma_j\}_{j\geq 1}$ of $\Q$ with $\sigma_j \leq \sigma_{j+1}$ and $\sigma_j \to 0$ 
as $j \to \infty$.

By the Cauchy Interlacing Theorem (\cite[Proposition 12.5]{KSchbook}), we have
\[\sigma_i\leq \lambda_i, \ \forall 1\leq i \leq n,\]
where $\lambda_1\leq \lambda_2 \leq \cdots \leq \lambda_n\leq 0$ are the eigenvalues of $P_{L}\circ \Q$.

Since $\{A_i\}_{n\leq i \leq (2n-1)}$ is orthonormal, the trace of $P_{L}\circ \Q$ is
\beq
Trace(P_{L}\circ \Q)=\sum_{n\leq i \leq (2n-1)}\Q(A_i,A_i)=\sum_{1\leq i \leq n}\lambda_i.
\eeq
By the part (iii) of Lemma \ref{4-5},  we have: 
\[ \lambda_1 \geq -16 \sqrt{\frac{3}{\pi}}.\]

By Proposition \ref{4-4},

\begin{eqnarray*}
-2^{-30}\cdot n &\geq& \sum_{n\leq i \leq (2n-1)}\Q(A_i,A_i) =  \sum_{1\leq i \leq n}\lambda_i\\
&=&\sum_{i=1}^{[\sqrt{n}]-1}\lambda_i+\sum_{[\sqrt{n}]}^n \lambda_i\\
&\geq & ([\sqrt{n}]-1) \lambda_{1}+(n-[\sqrt{n}]+1)\lambda_{[\sqrt{n}]}\\
&\geq& ([\sqrt{n}]-1) (-16 \sqrt{\frac{3}{\pi}})+(n-[\sqrt{n}]+1)\lambda_{[\sqrt{n}]}.
\end{eqnarray*}

Divided by $n$ for the inequality above and for large $n>>1$, we have
\[\lambda_{[\sqrt{n}]}\leq -2^{-31}.\]

Since $ \lambda_i$ is increasing,
\[\lambda_{j}\leq -2^{-31}, \ \forall \ 1\leq j \leq [\sqrt{n}].\]

Therefore, 

\[\sigma_{j}\leq -2^{-31}, \ \forall \ 1\leq j \leq [\sqrt{n}].\]

Let $n \to \infty$, this contradicts with the fact that $\sigma_j \to 0$ as $j \to \infty$.
\end{proof}

\section{Twisted Harmonic map into $T_H(1)$}
Harmonic maps theory is an important topic in geometry and analysis. In this section, we consider harmonic maps into the {\UTS} 
with the {\WP} metric. With newly obtained curvature information about $T_H(1)$, we study harmonic maps into $T_H(1)$ and 
prove some rigidity results. We follow a similar argument as the proof of Theorem 1.2 in \cite{Wu14}. 

\begin{proof}[Proof of Theorem \ref{mt-4}]
Since the {\co} on $T_H(1)$ is non-positive definite, $T_H(1)$ also has non-positive Riemannian sectional curvature in the 
complexified sense as stated in \cite{MSY93}. Suppose that $f$ is not constant. From \cite[Theorem 2]{MSY93} (one may also see 
\cite{Cor92}), we know that $f$ is a totally geodesic immersion. We remark here that the target space in \cite[Theorem 2]{MSY93} is stated to be a finite dimensional complex manifold. Actually, the proof goes through in the case that the target space has infinite dimension, without modification. Similar arguments are applied in \cite[Section 5]{Duc15}.

On quaternionic hyperbolic manifolds $H_{Q,m}=Sp(m,1)/Sp(m)$, since $f$ is totally geodesic, we identify the image 
$f(H_{Q,m}=Sp(m,1)/Sp(m))$ with $H_{Q,m}=Sp(m,1)/Sp(m)$. We may select $p\in H_{Q,m}$. Choose a quaternionic line 
$l_{Q}$ on $T_{p}H_{Q,m}$, and we may assume that $l_{Q}$ is spanned over $\R$ by $v,Iv,Jv$ and $Kv$. Without loss of generality, 
we may assume that $J$ on $l_{Q}\subset T_{p}H_{Q,m}$ is the same as the complex structure on $T_H(1)$.

Let $Q^{H_{Q,m}}$ be the curvature operator on $H_{Q,m}$, and we choose an element 
\beq
v\wedge Jv+Kv\wedge Iv \in \wedge^2 T_{p}H_{Q,m}.
\eeq
Then we have 
\begin{eqnarray*}
& Q^{H_{Q,m}}(v\wedge Jv+Kv\wedge Iv,v\wedge Jv+Kv\wedge Iv)=\\
&R^{H_{Q,m}}(v,Jv,v,Jv)+R^{H_{Q,m}}(Kv,Iv,Kv,Iv)+2\cdot R^{H_{Q,m}}(v,Jv,Kv,Iv).
\end{eqnarray*}
Since $I$ is an isometry, we have 
\begin{eqnarray*}
R^{H_{Q,m}}(Kv,Iv,Kv,Iv)&=&R^{H_{Q,m}}(IKv,IIv,IKv,IIv)\\
&=&R^{H_{Q,m}}(-Jv,-v,-Jv,-v)\\
&=&R^{H_{Q,m}}(v,Jv,v,Jv). 
\end{eqnarray*}
Similarly,
\begin{eqnarray*}
R^{H_{Q,m}}(v,Jv,Kv,Iv)&=&R^{H_{Q,m}}(v,Jv,IKv,IIv)\\
&=&R^{H_{Q,m}}(v,Jv,-Jv,-v)\\
&=&-R^{H_{Q,m}}(v,Jv,v,Jv).
\end{eqnarray*}
Combining the terms above, we have
\beq
Q^{H_{Q,m}}(v\wedge Jv+Kv\wedge Iv,v\wedge Jv+Kv\wedge Iv)=0.
\eeq
Since $f$ is a totally geodesic immersion,
\beq
Q^{T_H(1)}(v\wedge Jv+Kv\wedge Iv,v\wedge Jv+Kv\wedge Iv)=0.
\eeq
On the other hand, by Theorem \ref{mt-1}, there exists $E \in \wedge^2TT_H(1)$ such that
\beq
v\wedge Jv+Kv\wedge Iv=E-\textbf{J}\circ E.
\eeq
Hence,
\begin{eqnarray}\label{16}
\textbf{J}\circ(v\wedge Jv+Kv\wedge Iv) &=&\textbf{J}\circ(E-\textbf{J}\circ E) \notag \\
&=&\textbf{J}\circ E-\textbf{J}\circ\textbf{J}\circ E \notag\\
&=&\textbf{J}\circ E-E \notag\\
&=&-(v\wedge Jv+Kv\wedge Iv).
\end{eqnarray}
On the other hand, since \textbf{J} is the same as $J$ in $H_{Q,m}$, we also have
\begin{eqnarray}\label{17}
&&\textbf{J}\circ(v\wedge Jv+Kv\wedge Iv)=(Jv\wedge JJv+JKv\wedge JIv)\\
\nonumber &=&Jv\wedge(-v)+Iv\wedge(-Kv)=v\wedge Jv+Kv\wedge Iv.
\end{eqnarray}

From equations (\ref{16}) and (\ref{17}) we get
\begin{eqnarray*}
v\wedge Jv+Kv\wedge Iv=0
\end{eqnarray*} 
which is a contradiction since $l_{Q}$ is spanned over $\R$ by $v,Iv,Jv$ and $Kv$.

In the case of the Cayley hyperbolic plane $H_{O,2}=F_{4}^{20}/SO(9)$, the argument is similar by 
replacing a quaternionic line by a Cayley line \cite{Cha72}. 
\end{proof}

\begin{remark}
Since $T_H(1)$ has negative sectional curvature, any symmetric space of rank $\geq 2$ can not be totally geodesically immersed in $T_H(1)$. 
The argument in the proof of Theorem \ref{mt-4} shows that the rank one symmetric spaces $Sp(m,1)$ and $F_4^{-20}$ also can not be 
totally geodesically immersed in $T_H(1)$. It would be \textsl{interesting} to study whether the remaining two noncompact rank one symmetric 
spaces $\mathbb{H}^n$ and $\mathbb{CH}^n$ can be totally geodesically immersed in $T_H(1)$ (or $Teich(S)$).
\end{remark}

\bibliographystyle{amsalpha}
\bibliography{wp}
\end{document}